\documentclass[10pt]{amsart}

\usepackage{fullpage,amssymb,amsmath,amsthm,graphicx,mathrsfs,bm,hyperref,bm,mathtools,tikz,enumitem}
\usepackage{pgfplots,subcaption}
\usetikzlibrary{patterns}
\pgfplotsset{compat=1.18}
\usepackage[english]{babel}

\newtheorem{theorem}{Theorem}[section]

\newtheorem{corollary}[theorem]{Corollary}
\newtheorem{proposition}[theorem]{Proposition}
\newtheorem{lemma}[theorem]{Lemma}

\theoremstyle{definition}
\newtheorem{definition}[theorem]{Definition}
\newtheorem{remark}[theorem]{Remark}

\newcommand\R{\mathbb{R}}
\newcommand\Q{\mathbb{Q}}
\newcommand\Z{\mathbb{Z}}
\newcommand\N{\mathbb{N}}
\newcommand\1{\mathbf{1}}
\newcommand{\ssum}[1]{\sum\limits_{\substack{#1}}}
\newcommand{\e}{\mathrm{e}}
\newcommand{\eps}{\varepsilon}
\renewcommand{\hbar}{\overline{h}}
\newcommand{\deltabar}{\overline{\delta}}
\newcommand{\abs}[1]{\left| #1 \right|}
\newcommand{\df}{{\rm d}}

\newcommand{\cB}{{\mathcal B}}
\newcommand{\cF}{{\mathcal F}}
\newcommand{\cP}{{\mathcal P}}

\newcommand{\image}[2]{\begin{center}
\includegraphics[width=#1\textwidth]{#2.png}
\end{center}}

\newcommand{\LattE}{{\texttt{LattE}}}

\DeclareMathOperator{\card}{card}

\title{Computing sieve integrals using LattE, and the density of integers with a localized divisor}

\author{Sary Drappeau}
\address{Laboratoire de Mathématiques Blaise Pascal \\ Université Clermont Auvergne \\ Institut Universitaire de France \\ 3 place Vasarely \\ 63178 Aubière Cedex, France}
\email{sary.drappeau@uca.fr}
\thanks{}

\author{Adrien Mounier}
\address{Institut de Mathématiques de Marseille\\ Aix-Marseille Université \\
  163 Av. de Luminy\\
  13009 Marseille, France}
\email{adrien.mounier@univ-amu.fr}

\thanks{}

\begin{document}


\date{}

\begin{abstract}
  We consider the problem of estimating numerically integrals of the shape
  $$ \int_P \frac{\df t}{t_1 \dotsb t_k} $$
  where~$P \in \R_{>0}^k$ is a convex polytope, $t=(t_1,\dotsc, t_k)$ and~$\df t$ is the Lebesgue measure.
  This type of integral appears frequently in main terms of sieve theory.

  We propose a simple method, based on the \LattE{} software for integration of polynomials over polytopes,
  which computes rigorous bounds on this integral in polynomial time with respect to the precision (in bits).
  We test the method on several examples from the literature of sieve theory.

  We apply our results to compute numerical approximations to the natural density
  $$ h(\alpha, \beta) := \operatorname{density}\{n\in\N, \exists d\mid n, d\in [n^\alpha, n^\beta]\}, \qquad (0<\alpha<\beta<1) $$
  of integers having a localized divisor, in the region~$\beta - \alpha \geq 0.02$.
  One ingredient involved is a refined formula for~$h(\alpha, \beta)$ which involves a manageable number of terms for these~$\alpha, \beta$.
  As a corollary, we give a numerical approximation of the leading constant in a theorem of Haddad and Koukoulopoulos on the average of the logarithm of middle-divisors of integers.
\end{abstract}



\thanks{This work was partially funded by ANR-20-CE91-0006 / FWF-I-4945-N, ANR-24-CE93-0016 / FNS-10.003.145}

\maketitle

\section{Introduction}

\subsection{Sieve integrals}

The present paper is concerned with the numerical evaluation of integrals of the type
\begin{equation}
  I(P) := \int_P \frac{\df t}{t_1 t_2 \dotsb t_k},\label{eq:integral-prototype}
\end{equation}
where~$P \subset \R_{>0}^k$ is a convex polytope, $t = (t_1, \dotsc, t_k)$ and~$\df t$ denotes the Lebesgue measure on~$\R^k$.
Changing variables, this can also be seen as the volume of the convex set~$\log P$ (which is not a polytope in general).
This family of integrals is ubiquitous in sieve theory, where sums over prime numbers of the shape
$$ S(P, x) := \ssum{p_1, \dotsc, p_k \\ (*)} \frac1{p_1 \dotsb p_k} $$
naturally arise as sieve upper- and lower-bounds; here~$(*)$ denotes a collection of inequalities involving products of the~$p_j$, which can be typically encoded as a disjunction of conditions of the shape
$$ \Big(\frac{\log p_1}{\log x}, \dotsc, \frac{\log p_j}{\log x}\Big) \in P $$
where~$P$ is a convex polytope in~$\R_{>0}^k$ as above. By Mertens' theorem (see \emph{e.g.}~\cite[Lemma~9]{Tenenbaum1980}), we have~$S(P, x) \to I(P)$ as~$x\to \infty$. Occurrences of integrals of the shape~\eqref{eq:integral-prototype} in standard textbooks can be found in:
\begin{itemize}
  \item Halberstam-Richert's ``Sieve methods''~\cite{HalberstamRichert1974}: p.243, p.322;
  \item Friedlander-Iwaniec's ``Opera de cribro''~\cite{FriedlanderIwaniec2010}: p.237, p.321, p.359, p.444, p.484;
  \item Harman's ``Prime-detecting sieves''~\cite{Harman2007}: pp.58-59, p.74, pp.88-90,
\end{itemize}
among many other sources.

The most fundamental question in sieve theory is to devise a sieving procedure to isolate primes, or almost-primes, in a sequence, given its distribution in arithmetic progressions (say). The main terms arising from the sieving procedure often take the shape of~$S(P, x)$ for some~$P$, and there is therefore need for rigorous and reasonably fast methods of numerical computations of integrals of the shape~\eqref{eq:integral-prototype}.
In some cases, such as the $\beta$-sieve weight's upper and lower density functions~$f_\kappa, F_\kappa$~\cite[Chapter~12]{FriedlanderIwaniec2010} (see also~\cite{Mounier2025} for another example involving integers free of large prime divisors), the structure of such integrals can be used to provide the desired quantities as solutions of differential-delay equations, which can be solved numerically to high precision~\cite{MarsagliaEtAl1989}, however in most other cases, and in particular in works using the Harman sieve weights, there is no clear iterative structure to exploit.

The aim of the present paper is to propose a way to compute integrals of the shape~$I(P)$, which is:
\begin{itemize}
  \item Rigorous: it relies on interval arithmetic and the \LattE{} software for exact integration of rational polynomials on rational polytopes~\cite{BaldoniEtAl2011,DeLoeraEtAl2011,DeLoeraEtAl2013},
  \item Polynomial-time in bit precision: we will show that it executes in time~$O(n^C)$ where~$n$ is the desired precision in bits,
  \item Open-source: it is provided as a Python script on GitHub~\cite{github-sieve-integral} using Sage libraries~\cite{SageMath}.
\end{itemize}

\begin{theorem}\label{th:poly-complexity}
  Given~$P\subset\R_{>0}^k$ a fixed polytope with equations defined over~$\Q$, the integral~$I(P)$ can be approximated with precision~$2^{-n}$ in time~$O(n^{O(1)})$.
\end{theorem}

Besides this result, which is mainly of theoretical interest, we have implemented and tested the method on several examples of the sieve literature.

\subsection{Density of integers with a localized divisor}

We also applied our method to provide a numerical computation of the asymptotic density of the set of integers having a localized divisor,
\begin{equation}
  \begin{cases} h(\alpha, \beta) := \lim_{x\to \infty} \frac{H(x, x^\alpha, x^\beta)}x , \\ H(x, y, z) := \card\{n\in [1, x], \exists d\mid n, d\in [y, z]\}, \end{cases} \qquad (0\leq \alpha \leq \beta \leq 1).\label{eq:def-hab}
\end{equation}
For the existence of~$h(\alpha, \beta)$ we refer to~\cite{Tenenbaum1980}.
The asymptotic behaviour of~$H(x, y, z)$ is a recurring topic in multiplicative and probabilistic number theory.
A famous problem of Erd\H{o}s, known as the multiplication table problem, asks for the asymptotic behaviour of~$A(x) := H(x, \sqrt{x}, 2\sqrt{x})$, and has been the topic of landmark contributions by Erd\H{o}s~\cite{Erdos1960}, Tenenbaum~\cite{Tenenbaum1984}, Ford~\cite{Ford2008}, and very recently Green and Sawhney~\cite{GreenSawhney2026} who have announced an asymptotic equivalent for~$A(x)$.

Here we are concerned with the numerical estimation of~$h(\alpha, \beta)$ for fixed~$\alpha<\beta$.
Tenenbaum~\cite[pp.~31--33]{Tenenbaum1980} showed that~$h(\alpha, \beta)$ can be approximated by a finite sums of integrals of the shape~\eqref{eq:integral-prototype}. The complexity of these integrals (their dimensions, notably) grows as~$\beta-\alpha$ decreases.
In the special case~$\beta = 1/2$ and~$\alpha \leq 2/5$, it is proven in~\cite[Théorème~B.(i)]{Tenenbaum1980} that
$$ h(\alpha, 1/2) = 1 - \log\frac1{1-\alpha} - I(P_\alpha), $$
where
\[ P_\alpha =
  \begin{cases}
    \varnothing & (\alpha \leq 1/3), \\
    \{(t_1, t_2, t_3)\in \R_{>0}, 1-\alpha - t_2 < t_1 < t_2 < t_3 < 1 - t_1 - t_2 - t_3\} & (1/3 \leq \alpha \leq 2/5).
  \end{cases} \]
The first case ($\alpha \leq 1/3$) is the only case when~$h(\alpha, \beta)$ possesses an explicit expression as sums of elementary functions. Note that this special case~$\beta=1/2$ relies on divisor symmetry, since~$h(\alpha, 1/2) = h(\alpha, 1-\alpha)$ (which doubles the effective difference $\beta-\alpha$).

Very recently, Haddad~\cite{Haddad2026} obtained an expression of~$h(\alpha, \beta)$ as a genuinely finite sum of integrals of the shape~\eqref{eq:integral-prototype} of dimension~$k = \lfloor 1/(\beta-\alpha)\rfloor$, which involves~$2^{2^k}$ polytopes.

We provide an expression of~$h(\alpha, \beta)$ which involves a reduced number of polytopes, so that these integrals become manageable for~$k\leq 7$.
Along with a rigorous bound on the contribution of large~$k$ contribution, this permits a numerical estimation of~$h(\alpha, \beta)$ in a wide range of parameters.

\begin{theorem}\label{th:hab-formula-vague}
  Let~$0<\alpha<\beta<1$ be fixed. The quantity~$h(\alpha, \beta)$ can be expressed as a sum
  \begin{equation}
    h(\alpha, \beta) = \sum_{0\leq k \leq 1/(\beta-\alpha)} \sum_{P \in \cP_k(\alpha, \beta)} J_{\alpha, \beta}(P),\label{eq:expr-h-polytopes-rough}
  \end{equation}
  where~$J_{\alpha, \beta}(P)$ is a~$k$-dimensional integral similar to~$I(P)$, and~$\card{\cP_k}\leq 2 n_k$, where
  \begin{equation}
    n_2 = 1, \qquad n_3 = 4, \qquad n_4 = 16, \qquad n_5 = 91, \qquad n_6 = 1053, \qquad n_7 = 43141.\label{eq:nk}
  \end{equation}
\end{theorem}
The main point is the size of~$n_k$, which makes this formula useable in practice.
The precise statement is given in Theorem~\ref{th:expr-h} below, and includes an effective bound on the contribution of large~$k$.

We used this formula and implemented our method of computation of~$I(P)$ to compute~$h(\alpha, \beta)$ for all values of~$\alpha$, $\beta$ along a grid of step~$1/100$ for which~$\beta - \alpha \geq 0.02$.
The approximate surface~$(\alpha, \beta, h(\alpha, \beta))$ obtained is represented in Figure~\ref{fig:h-values}.
An interactive 3D plot representing an upper (blue) and lower (red) bound for~$h(\alpha, \beta)$ can be found in the HTML ancillary file in preprint repositories (HAL, arXiv) of this manuscript.

\begin{figure}[h!]
  \centering
  \begin{subfigure}[t]{.5\textwidth}
    \centering
    \image{.7}{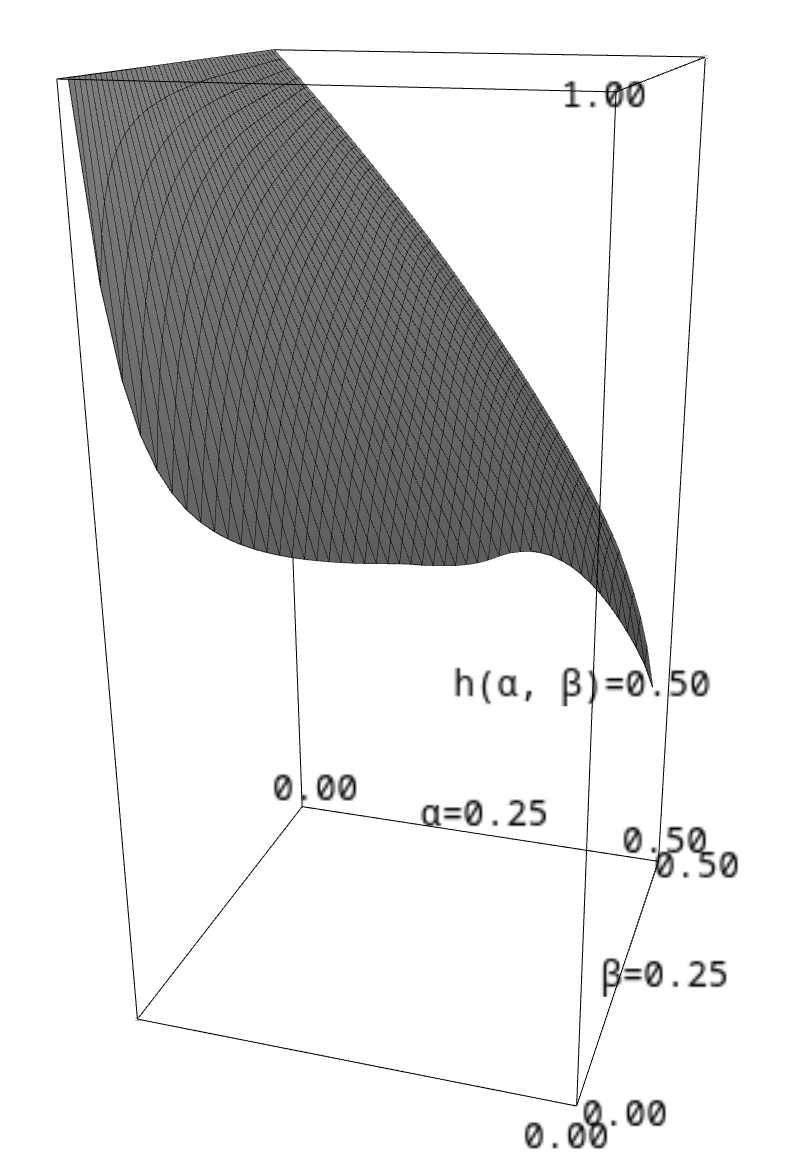}
  \end{subfigure}%
  ~
  \begin{subfigure}[t]{.5\textwidth}
    \centering
    \image{.8}{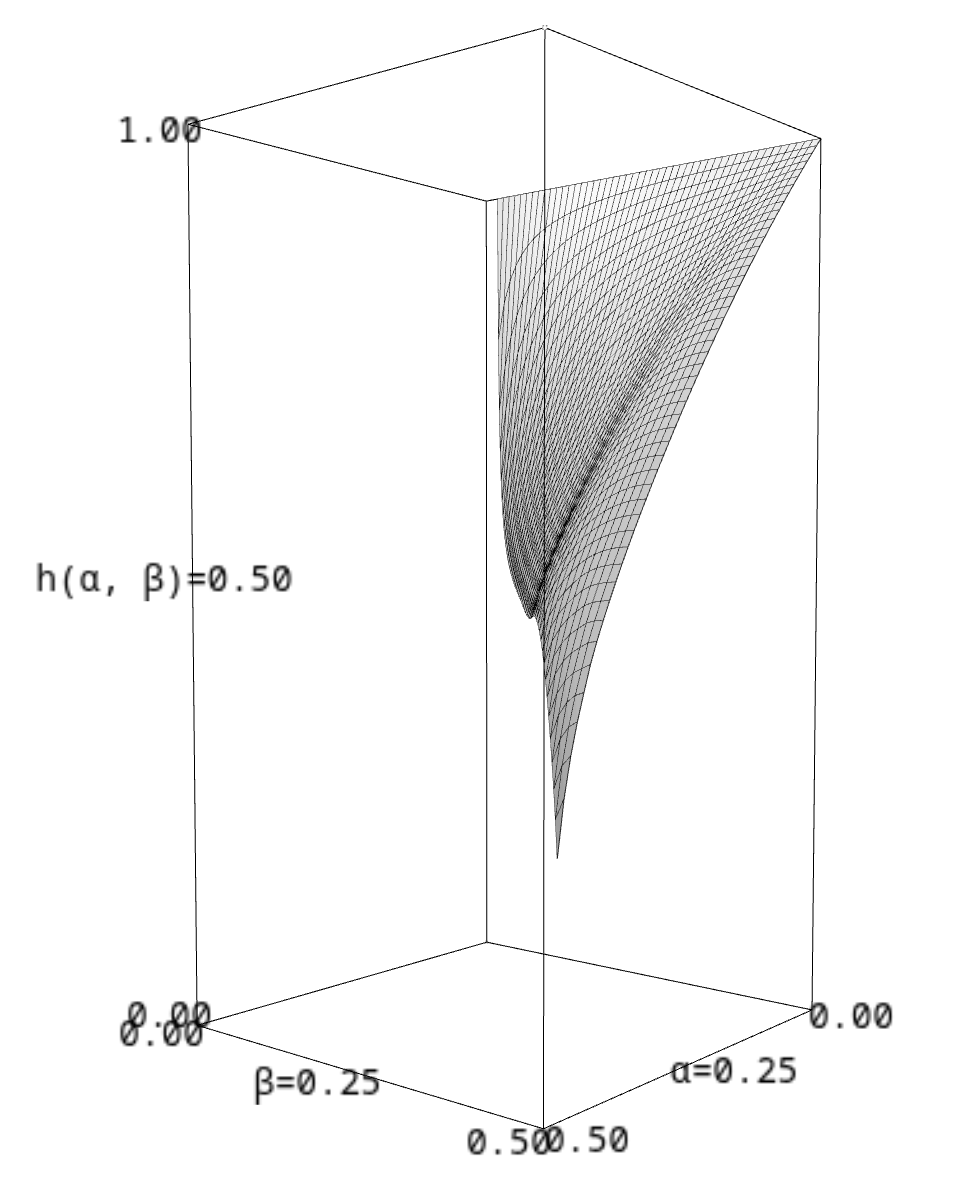}
  \end{subfigure}
  \caption{Approximate values of~$h(\alpha, \beta)$. A grid in~$(\alpha,\beta)$ of step~$1/100$ is projected onto the surface.}
  \label{fig:h-values}
\end{figure}

In the paper~\cite{Haddad2026}, Haddad also obtains an effective estimate comparing~$h(\alpha, \beta)$ with~$H(x, x^\alpha, x^\beta)$, which together with Ford's theorem~\cite{Ford2008} implies
$$ h(\alpha, \beta) \asymp (\beta - \alpha)^{\delta} (\log\tfrac1{\beta - \alpha})^{-3/2}, \qquad \delta = 1 - \frac{1 + \log\log 2}{\log 2} $$
for~$\alpha \gg 1$. Comparing with our graph, we note that even though our points~$(\alpha, \beta)$ seem relatively close to the diagonal, the values of~$h(\alpha, \beta)$ are still far from their value~$0$ \emph{at} the diagonal, due to the very small size of the exponent~$\delta = 0.086...$.

We provide more information on this computation in Section~\ref{sec:numer-estim-hab} below.

\medskip{}

To illustrate this formula we give a numerical approximation to the leading constant in the following result. For all~$n\geq 1$ let~$d^*(n)$ be least divisor of~$n$ larger than~$\sqrt{n}$,
$$ d^*(n) := \min\{d\mid n, d\geq \sqrt{n}\}. $$
In the paper~\cite{HaddadKoukoulopoulos2025}, Haddad and Koukoulopoulos obtained, as a simple corollary of their result on the optimal rate of convergence in Billingsley's theorem, the estimate\footnote{Note that it is the error term~$O(x)$ in~\eqref{eq:HK-dstar} which is the difficult part, and explains the relevance with the tools developed in~\cite{HaddadKoukoulopoulos2025}.}
\begin{equation}
  \sum_{n\leq x} \log d^*(n) = c x \log x + O(x),\label{eq:HK-dstar}
\end{equation}
where~$c>0$ is the constant given by
$$ c = 1 - \int_0^{1/2} h(\alpha, \tfrac12) \df \alpha, $$
see the end of Section~2 in~\cite{Haddad2026}. Using the formula above we can bound~$c$ as follows.

\begin{corollary}\label{cor:mid-div}
  We have~$0.651 \leq c \leq 0.666$.
\end{corollary}

To prove this estimate, instead of simply sampling~$\alpha$, we take advantage of the fact that the inequalities of the polytopes~$P$ in~\eqref{eq:expr-h-polytopes-rough} involve linear combinations of~$\alpha, \beta$, so that the~$\alpha$-integral can be treated as an additional polytope integration dimension.

\subsection{Notations}

We use interchangeably the Landau and Vinogradov notation~$X = O(Y)$ or~$X \ll Y$ if~$\abs{X} \leq C \abs{Y}$ for some constant~$C$.
The notation~$X\asymp Y$ means that~$X \ll Y$ and~$Y \ll X$.
The number~$C$ may depend on various parameters (such as the dimension~$k$ or the polytope~$P$ in the time complexity estimates).
This will be clear from the context.
The letter~$p$ denotes a prime number. For a natural number~$n$, the notation~$P^+(n)$ and~$P^-(n)$ denote the largest and smallest prime divisor of~$n$, respectively, with the convention~$P^+(1)=1$ and~$P^-(1)=\infty$.

\section{Computing sieve integrals}\label{sec:comp-sieve-integr}

In this section, we fix~$k \geq 1$ and a convex polytope~$P\subset \R_{>0}^k$ defined by linear inequalities with coefficients in~$\Z$. We consider~$P$ to be fixed with coefficients of size~$O(1)$.
Our aim is to compute the integral
$$ \int_{t\in P} \frac{dt}{t_1 \dotsb t_k} $$
with a precision~$2^{-n}$. We focus on the time complexity with respect to~$n$ and~$k$.

\subsection{Overview}

The \LattE{} integration software~\cite{BaldoniEtAl2011,DeLoeraEtAl2011,DeLoeraEtAl2013} evaluates exactly integrals of the shape
$$ \int_P f(t) dt $$ 
for polynomials~$f$ with coefficients in~$\Q$. Assuming that~$f$ has degree~$m$ with coefficients of size~$m^{O(1)}$, all else being fixed, the time complexity is polynomial in~$m$.
We refer to~\cite{DeLoeraEtAl2013} for more details about this algorithm, and in particular its relation to the work of Brion~\cite{Brion1988}, Barvinok~\cite{Barvinok1991}, Lasserre~\cite{Lasserre1983} and~Lawrence~\cite{Lawrence1991}.

The method we propose is simply to insert a Taylor expansion of~$1/(t_1 \dotsb t_k)$, and execute termwise integration using \LattE{}.
For this method to be numerically efficient, it is necessary that the variables~$t_j$ vary over an interval~$[T_0, T_1]$ with~$T_1/T_0$  close to~$1$. 
\begin{definition}
  Given a convex polytope~$P\subset \R_{>0}^k$, and a number~$\eta>0$, we say that~$P$ is~$\eta$-balanced if~$t, t'\in P$ implies that~$\forall j$, $t_j/t_j' \leq \e^\eta$.
\end{definition}
In other words, this means that~$\log P$ is contained in a box parallel to the axes with edges lengths at most~$\eta$.
Thus our method decomposes into three steps:
\begin{enumerate}
  \item Split the polytope~$P$ into a union~$\cup_i P_i$ of~$\eta$-balanced polytopes.
  \item In each piece~$P_i$, Taylor expand~$1/(t_1\dotsb t_k) \approx f(t)$, where~$f$ is a polynomial, with a rigorous estimation on the error.
  \item Compute the integral~$\int_{P_i} f$ exactly using \LattE{}.
\end{enumerate}

The method depends on the choice of parameter~$\eta$ and decomposition~$\cup_i P_i$.

The method proposed here is not specific to the special shape of the integrand (here~$1/t_1 \dotsb t_k$), so long as one has a rigorous bound on the tail of its Taylor expansion (Lemma~\ref{lem:taylor-expansion} below).

\subsection{Description of the method}

\subsubsection{Splitting the polytope}\label{sec:polytope-splitting}

We first choose a parameter~$\eta>0$, and we look for a partition~$P = \cup P_i$ into~$\eta$-balanced polytopes.
We make the simple choice of taking~$P_i$ of the shape
\begin{equation}
  P \cap ([u_1, \e^\eta u_1]\times \dotsb \times [u_k, \e^\eta u_k]).\label{eq:Delta-splitting-part}
\end{equation}
The number of these~$P_i$ is~$O(\eta^{-k})$. We thus reduce the computation of~$I(P)$ to~$O(\eta^{-k})$ many computations of~$I(P_i)$ where~$P_i$ which is~$\eta$-balanced.

We note that with this method, the number of vertices of~$P_i$ may become exponential in~$k$ even if the original polytope~$P$ had few vertices to begin with.

\subsubsection{Taylor expansion}

Suppose next that~$P$ is~$\eta$-balanced. More precisely, write
$$ P \subset [c_1 - r_1, c_1 + r_1] \times \dotsc [c_k - r_k, c_k + r_k] $$
for numbers~$c_i, r_i > 0$ which, by hypothesis, we can pick subject to~$(c_j + r_j) / (c_j - r_j) \leq \e^\eta$, which is to say~$\frac{r_j}{c_j} \leq \varrho$ where
\begin{equation}
  \varrho := \frac{\e^\eta - 1}{\e^\eta + 1}.\label{eq:def-varrho}
\end{equation}

\begin{lemma}\label{lem:taylor-expansion}
  Let~$\eta>0$ and define~$\varrho$ as in~\eqref{eq:def-varrho}. Let~$P\subset \R_{>0}^k$ be~$\eta$-balanced.
  For all~$t = (t_j) \in P$, write~$u_i := (c_i - t_i) / c_i$, so that~$\abs{u_i} \leq r_i/c_i \leq \varrho$. Then we have
  $$ \frac1{t_1 \dotsc t_k} = \Big(\prod_{i=1}^k \frac1{c_i}\Big)\times \Big( \ssum{j_1, \dotsc, j_k \geq 0 \\ j_1 + \dotsb + j_k \leq N} \prod_{i=1}^k u_i^{j_i} + \eps\Big), $$
  where
  $$ \abs{\eps} \leq \varrho^N \sum_{0\leq j < k} \binom{N+k}{j} \Big(\frac{\varrho}{1-\varrho}\Big)^{k-j}. $$
\end{lemma}
\begin{proof}
  The error term is bounded in absolute value by
  \[
    \abs{\sum_{j_1 + \dotsb + j_k > N} \prod_{i=1}^k u_i^k} \leq \sum_{\ell>N} \varrho^\ell \sum_{j_1 + \dotsb + j_k = \ell} 1
    = \sum_{\ell > N} \varrho^\ell \binom{\ell+k-1}{\ell}
    = \frac1{(k-1)!} F^{(k-1)}(\varrho),
  \]
  where~$F(\varrho) = \varrho^{N+k}/(1-\varrho)$. Thus
  $$ F^{(k-1)}(\varrho) = \sum_{j=0}^{k-1} \binom{k-1}{j} \times \frac{(N+k)!\varrho^{N+k-j}}{(N+k-j)!} \times \frac{(k-1-j)!}{(1-\varrho)^{k-j}} $$
  and the result follows after simplification of the factorials.
\end{proof}

We use this formula to provide a rigorous bound on the truncation error in the Taylor expansion.

\subsubsection{Computation of the main term}

Following the previous two steps, we are now to evaluate about~$O(\eta^{-k})$ many integrals of the shape
\begin{equation}
  \int_P f_{k,N}(t) \df t\label{eq:integral-polynomial}
\end{equation}
where~$f_{k,N}(t)$ is, after appropriate dilation and translation, the sum of the complete symmetric polynomials of degree~$\leq N$,
$$ f_{k,N}(t) = \sum_{0\leq m \leq N} h_m(t), \qquad h_m(t) = \ssum{j_1, \dotsc, j_k \geq 0 \\ j_1 + \dotsb + j_k = m} \prod_{i=1}^k t_i^{j_i}. $$
The \LattE{} software~\cite{BaldoniEtAl2011,DeLoeraEtAl2011,DeLoeraEtAl2013} computes the integral~\eqref{eq:integral-polynomial} exactly.
To achieve this, \LattE{} proceeds in three steps:
\begin{enumerate}[label=(\roman*)]
  \item it decomposes the polytope~$P$, either by simplicial triangulation, or as intersection of tangent cones,
  \item it decomposes the polynomial~$f$ into a linear combination of powers of linear forms,
  \item it integrates each power of linear form using a closed form formula.
\end{enumerate}

\subsubsection*{Step (i)}

Recall that we chose to partition our original polytope~$P$ into parts of shape given in~\eqref{eq:Delta-splitting-part}. In practice, the original polytope~$P$ is defined by a small number (linear in~$k$) of hyperplanes, and thus we call~\LattE{} on a polytope which looks like a hypercube cut out by a small number of hyperplanes.
On such polytopes, the splitting of~$P$ into simplexes will produce a splitting into~$\approx k!$ many small polytopes. By contrast, the tangent cone decomposition method requires~$O(\e^{O(k)})$ many terms. Although we mainly consider~$k$ to be fixed, this indicates that in practice the tangent cone decomposition performs better, and this is what we observed indeed in our later numerical computation.

\subsubsection*{Step (ii)}

Step~(ii) raises the question of the computation of the Waring rank of the polynomial~$h_m(t)$.
\LattE{} uses a monomial identity decomposition~\cite[eq.~(13)]{BaldoniEtAl2011} which gives an expression of~$h_m(t)$ as a sum of~$O(m^k)$ powers of linear forms for~$k$ fixed. 

We take the opportunity here to record the following decomposition, which involves a sum over~$O(m^{k-1})$ terms with rational coefficients and thus saves one factor of~$m$.

\begin{lemma}\label{lem:waring}
  For all~$m\geq 1$ and $0\leq r, a \leq n$, let~$L_{r, a}^{(n)}$ be the coefficients in the Lagrange polynomials
  $$ \sum_{r=0}^m L_{r, a}^{(m)} X^r = \prod_{\substack{0\leq b \leq m \\ b \neq a}} \frac{X - b}{a - b}. $$
  and for all~$0\leq a_1, \dotsc, a_{k-1} \leq m$ define
  \begin{equation}
    C_m(a_1, \dotsc, a_{k-1}) = \ssum{r_1, \dotsc, r_{k-1}\geq 0 \\ r_1 + \dotsb + r_{k-1} \leq m \\ r_k := m - r_1 - \dotsb - r_{k-1}} \binom{m}{r_1, \dotsc, r_k}^{-1} \prod_{j=1}^{k-1} L_{r_j, a_j}^{(m)}.\label{eq:def-Cn}
  \end{equation}
  For all~$n\geq 0$ and~$t\in \R^k$, we have
  \begin{equation}
    h_m(t) =  \sum_{0\leq a_1, \dotsc, a_{k-1} \leq m} C_m(a_1, \dotsc, a_{k-1}) (a_1 t_1 + \dotsb + a_{k-1} t_{k-1} + t_k)^m \label{eq:decomp-hm-waring}
  \end{equation}
\end{lemma}
\begin{proof}
  We start with the equality
  \begin{equation}
    \sum_{a=0}^m L_{r, a}^{(m)} a^\ell = \1(\ell = r), \qquad (0 \leq \ell, r \leq m),\label{eq:vandermonde}
  \end{equation}
  which is obtained by inverting a Vandermonde matrix. Let~$r_j\geq 0$ satisfy~$r_1 + \dotsb + r_k = m$. Applying~\eqref{eq:vandermonde} at~$r = r_1, \dotsc, r_{k-1}$, we get
  $$ \sum_{0\leq a_1, \dotsc, a_{k-1} \leq m} \Big(\prod_{j=1}^{k-1} L^{(m)}_{r_j, a_j}\Big) (a_1 t_1 + \dotsb + a_{k-1} t_{k-1} + t_k)^m = \binom{m}{r_1, \dotsc, r_k} t_1^{r_1} \dotsb t_k^{r_k}. $$
  The statement follows upon summing over~$r_j$.
\end{proof}

The main difference between~\eqref{eq:decomp-hm-waring} and~\cite[eq.~(13)]{BaldoniEtAl2011} is explained as follows: in the expansion of the polynomials~$(t_1 + \dotsb + t_k)^m$ into monomials of the shape~$t_1^{j_1} \dotsb t_k^{j_k}$, we isolate a specific monomial~$t_1^{r_1} \dotsb t_k^{r_k}$ by detecting the conditions~$j_1 = r_1, \dotsc j_{k-1} = r_{k-1}$ among all exponents~$0\leq j_i \leq m$, which requires~$n^{k-1}$ sums; the equality~$j_k = r_k$ follows by homogeneity. The identity~\cite[eq.~(13)]{BaldoniEtAl2011} isolates the monomial~$t_1^{r_1} \dotsb t_k^{r_k}$ by detecting instead the conditions~$j_1 \leq r_1, \dotsc, j_k \leq r_k$, which involves~$r_1 \dotsb r_k$ sums. Equality everywhere follows by homogeneity. For a single monomial and small~$m$, the formula~\cite[eq.~(13)]{BaldoniEtAl2011} is much more efficient (it requires~$O(r_1 \dotsb r_k) = O(m + (m/k)^k)$ sums), but for the sum~$h_m(t)$, containing many monomials of degree~$m$, the formula~\eqref{eq:decomp-hm-waring} is more efficient in theory and recovers the optimal bound for single monomials~\cite[Proposition~3.1]{CarliniEtAl2012}.

We found, however, that in all applications we make of our method here, the values of~$m$ are small enough that the method of~\cite{BaldoniEtAl2011}, implemented efficiently in C by the \LattE{} software contributors, is very fast and requires only computations of binomial coefficients (leading terms of Lagrange polynomials) which are easier to implement. Therefore Lemma~\ref{lem:waring} was not implemented in our numerical tests. However it could prove useful in cases where~$m$ is large.

\subsubsection*{Step (iii)}

Finally, by~\cite{DeLoeraEtAl2011,DeLoeraEtAl2013} the \LattE{} computation is polynomial in~$N$ for fixed~$k$ and~$P$, and yields an exact result.

\subsection{Estimated time complexity}

We justify here the polynomial time complexity (Theorem~\ref{th:poly-complexity}). We consider~$k$ and~$P$ to be fixed and are looking for a result precise up to~$n$ bits, $n\geq 1$.

\begin{itemize}
  \item In the balancing step in Section~\ref{sec:polytope-splitting}, we reduce the computation to~$O(\eta^{-k})$ computations over balanced polytopes.
  \item In the Taylor approximation step, recalling the bound in Lemma~\ref{lem:taylor-expansion} and the approximation~$\varrho \asymp \eta$, we are to chose~$N \asymp n / \log(1/\eta) $ in order to guarantee~$n$ bits of precision.
  \item The \LattE{} computation for the degree~$N$ Taylor approximation executes in time~$O(N^C)$ for some~$C>0$.
\end{itemize}

The total computation time is therefore at most
\begin{equation}
  \ll \eta^{-k} \Big(\frac{n}{\log(1/\eta)}\Big)^{C}.\label{eq:comp-time-estimate-eta}
\end{equation}
For fixed~$k$ and varying~$n$, choosing~$\eta\gg 1$ proves the polynomial time complexity (Theorem~\ref{th:poly-complexity}).
We see from the expression~\eqref{eq:comp-time-estimate-eta} that it is not beneficial to choose~$\eta$ any smaller.

We have tested our method on a random polytope for which the coordinates vary within a factor at most~$5$. For a fixed~$P$ we tested several values of~$\eta$ and retrieved the average computation time (3 runs). The results for three polytopes of dimension~$3, 4$ and~$5$ respectively are displayed in Figure~\ref{fig:test-eta}. This illustrates the fact that any choice of~$\eta$ around~$0.4$ yields good results for these polytopes. This is in accordance with the numerical tests carried out in concrete cases of sieve theory which we will detail below.
One should keep in mind that the best value of~$\eta$ could vary with~$k$ or the geometry of~$P$.

\begin{figure}[h]
  \centering
  \begin{minipage}{0.33\textwidth}
    \centering
    \image{.9}{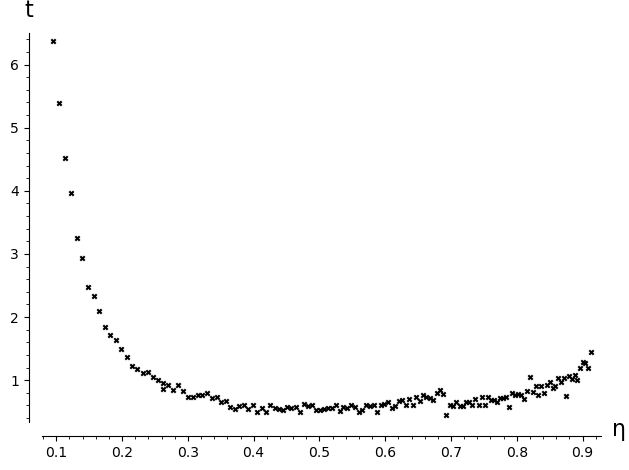}
  \end{minipage}%
  \begin{minipage}{0.33\textwidth}
    \centering
    \image{.9}{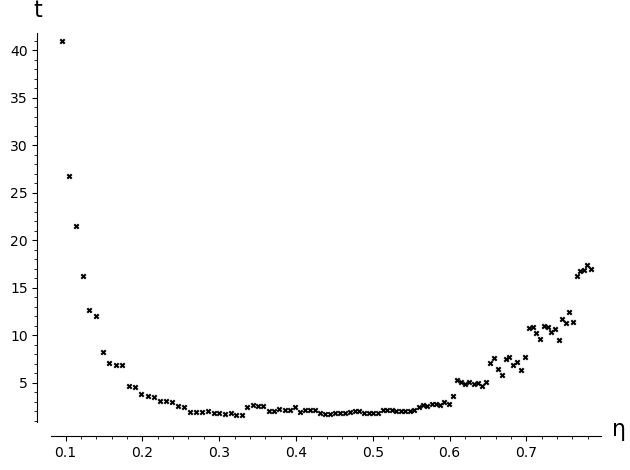}
  \end{minipage}%
  \begin{minipage}{0.33\textwidth}
    \centering
    \image{.9}{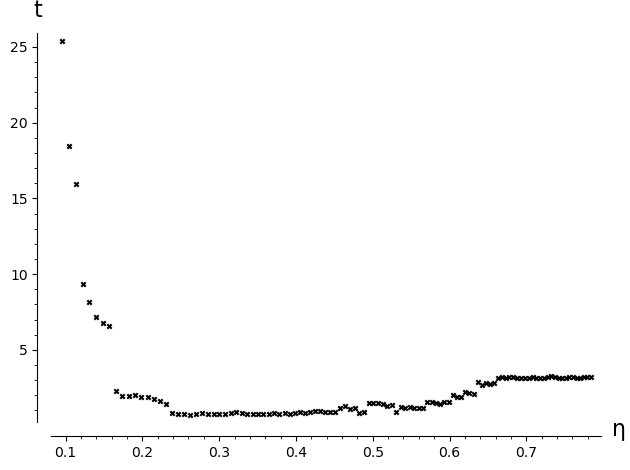}
  \end{minipage}
  \caption{Computation time (s) as~$\eta$ varies (dimension~$3$, $4$, $5$)}
  \label{fig:test-eta}
\end{figure}

\subsection{Variants}

\subsubsection{Integrals restricted to $\sum_j t_j = 1$}

In some cases, notably when counting almost-prime integers, the integral to be computed takes the shape
\begin{equation}
  I^*(P) = \int_{P \cap H_k} \frac{\df t_1 \dotsb \df t_{k-1}}{t_1 \dotsb t_k},\label{eq:def-I*}
\end{equation}
where
\begin{equation}
  \label{eq:def-Hk}
  H_k = \{t_1 + \dotsb + t_k = 1\}
\end{equation}
This is the case, for instance, in Ford-Maynard's recent work~\cite{FordMaynard2024} on prime-producing sieves.

The method we describe works equally well to compute~$I^*(P)$, expressing the variable~$t_k$ in the denominator in terms of the other variables. However in step~(i) (the polytope decomposition), one should be careful to balance the polytope along all~$k$ dimensions to ensure that the Taylor expansion is effective.

Another method, which is slower in practice but easier to implement, is to write
$$ I^*(P) = \int_{P \cap H_k} (t_1 + \dotsb + t_k) \frac{\df t_1 \dotsb \df t_{k-1}}{t_1 \dotsb t_k} =  \sum_{j=1}^k I(P_j), $$
where~$P_j\subset\R_{>0}^{k-1}$ is the projection of~$P$ on all axes except the~$j$-th. We used this alternate method to double-check our results.

\subsubsection{Integrals involving the Buchstab function}

The Buchstab function~\cite{Bukhshtab1937}~$\omega:\R_{>0} \to \R$ is the unique map continuous except at~$1$, differentiable outside of~$\{1, 2\}$ satisfying
\begin{align}
  & \omega(u) = 0, \qquad (0<u\leq 1), \notag \\
  & \omega(u) = 1/u, \qquad (1 < u \leq 2), \notag \\
  & (u\omega(u))' = \omega(u-1), && (u < 2). \label{eq:dde-omega}
\end{align}
This map occurs as the density of integers with large prime divisors. It is convenient to denote
\begin{equation}\label{eq:def-Omega}
  B(u) := u \omega(u), \qquad (u\in \R).
\end{equation}
For~$u>0$, we have
\begin{equation}
  \card\{ n\leq x, p\mid n \implies p > x^{1/u}\} \sim  B(u) \frac{x}{\log x}, \qquad (x\to \infty),\label{eq:Omega-asympt}
\end{equation}
see~\cite[Lemma~12.1]{FriedlanderIwaniec2010} or~\cite[Theorem~III.6.4]{Tenenbaum2015}.
As such, this map appears relatively frequently in inclusion-exclusion processes where some almost-prime number is left unconstrained.
An important case is Harman's alternative sieve, which has given some of the strongest results in problems on finding primes in various sets under type~I and type~II assumptions. We mention Dirichlet's theorem for prime denominators~\cite{Harman1983,Matomaeki2009}, or Chebyshev's problem on the largest prime divisor of~$n^2+1$~\cite{Merikoski2023,GrimmeltMerikoski2025}.

The main terms in Harman's alternate sieve method involve terms of the shape
\begin{equation}
  \int_{P\cap H_k} B\Big(\frac{\alpha_k}{\alpha_{k-1}}\Big) \frac{\df \alpha}{\alpha_1 \dotsb \alpha_k}\label{eq:Harman-type}
\end{equation}
where~$P\subset \R_{>0}^k$ is a polytope and~$H_k$ is the hyperplane defined in~\eqref{eq:def-Hk}, see~\cite[p.~58]{Harman2007}.
Estimating integrals of the shape~\eqref{eq:Harman-type} with our proposed method therefore raises the question of approximating the Buchstab function in the integrand in a way suitable to our method: either as a polynomial, or as integrals over polytopes.

\begin{enumerate}[label=(\alph*)]
  \item Polynomial approximations to the Buchstab function can be obtained very quickly by a method, devised in this context by Marsaglia-Zaman-Marsaglia~\cite{MarsagliaEtAl1989}, which consists in computing recursively Taylor expansions of~$B$ on intervals~$(n, n+1)$ ($n\in\N$), where it is real-analytic, using the delay-differential equation~\eqref{eq:dde-omega}. The case of Buchstab function itself is detailed in~\cite[section~7]{MarsagliaEtAl1989}.

  We found that an ad-hoc analysis of error terms in Taylor expansions of this kind, along with the Marsaglia-Zaman-Marsaglia method, is a practical and reasonably fast method for the computation of the integrals of type~\eqref{eq:Harman-type}.

  \item As an alternative route, a classical formula~\cite[p.~1241]{Bukhshtab1937} yields, for all~$u>0$, the expression
  \begin{equation}
    B(u) = \sum_{\ell\geq 0} I_\ell(u),\label{eq:expr-Omega-I}
  \end{equation}
  where~$I_0(u) = \1(u>1)$, and for all~$\ell\geq 1$,
  \begin{equation}\label{eq:def-I}
    I_\ell(u) = I(P_{\ell, u}),
    \qquad P_\ell := \Big\{t\in \R_{>0}^\ell, \begin{cases} t_\ell > 1, \\ \forall 1\leq j < \ell, t_j > t_{j+1} + 1, \\ t_1 < u-1. \end{cases} \Big\}.
  \end{equation}
  Note in particular that only~$\ell<u-1$ occur in the sum~\eqref{eq:expr-Omega-I}.
  By applying a dilation to the formula~\eqref{eq:def-I}, we write
  $$ I_\ell\Big(\frac{\alpha}{\beta}\Big) = \int_{P_\ell(\alpha, \beta)} \frac{\df t}{t_1 \dotsb t_\ell}, $$
  $$ P_\ell(\alpha, \beta) :=  \Big\{t\in \R_{>0}^\ell, \begin{cases} t_\ell > \beta, \\ \forall 1\leq j < \ell, t_j > t_{j+1} + \beta, \\ t_1 < \alpha - \beta. \end{cases} \Big\}. $$
  In this way, the integral~\eqref{eq:Harman-type} reduces to a finite number of integrals of the shape~$I^*(P)$ of growing dimension.
  This method is easy to implement, but the dimension growth makes it very slow in practice compared with the first method. We used it as a safety check to verify some of the numerical results (detailed below in Section~\ref{sec:exampl-using-harm} below) that were obtained using the first method.
\end{enumerate}

\subsubsection{Integrals involving the Dickman function}

The Dickman function~\cite{Dickman1930}~$\rho:\R_{>0} \to \R$ is the unique continuous map, differentiable everywhere except at~$1$, which satisfies
\begin{equation}\label{eq:def-rho}
  \begin{cases}
    \rho(u) = 1, & (0<u \leq 1), \\
    \rho'(u) = -\rho(u-1), & (u>1).
  \end{cases}
\end{equation}
It appears in number theory as the density of integers free of large prime divisors: for every~$u>0$, we have
\begin{equation}
  \card\{ n\leq x, p\mid n \implies p \leq x^{1/u}\} \sim  \rho(u) x, \qquad (x\to \infty),\label{eq:rho-asympt}
\end{equation}
see~\cite[Theorem~III.5.8]{Tenenbaum2015}.

In cases when the inclusion-exclusion principle is performed over integers which are not free of small prime divisors, one is often faced with a main term consisting of an integral of type~$I(P)$ weighted by the Dickman~$\rho$ function in the integrand. This will be the case for us when we study the density~\eqref{eq:def-hab} of integers with a localized divisor, see~\cite[Theorem~2]{Haddad2026}. We will study this case in detail in Section~\ref{sec:hab} below.

\medskip{}

As was the case for the Buchstab function, the iterative method of Marsaglia-Zaman-Marsaglia~\cite{MarsagliaEtAl1989} can be used to obtain Taylor approximations for~$\rho(u)$ very quickly. This method is used, for instance, in the SageMath~\cite{SageMath} implementation of the Dickman function by R.~Bradshaw. As for the Buchstab function, we found that this method produced satisfying results in our case.

\medskip{}

We also mention that, similarly to the Buchstab function, the Dickman function~$\rho$ can be expressed as a sum of nested integrals~\cite{Dickman1930},
\begin{equation}
  \label{eq:expr-rho-J}
  \rho(u) = \sum_{j=0}^\infty (-1)^\ell J_\ell(u),
\end{equation}
where~$J_0(u) = 1$ and for all~$\ell\geq 1$ and~$u>0$,
$$ J_\ell(u) = I(P_{\ell, u}), \qquad P_{\ell, u} = \{1 < t_k < \dotsb < t_1, t_1 + \dotsb + t_k < u\}. $$
However, contrarily to the Buchstab function expression~\eqref{eq:expr-Omega-I}, the individual integrals in~\eqref{eq:expr-rho-J} tend to grow with~$u$, whereas~$\rho(u)$ tends to~$0$ very quickly: the large amount of cancellation in the sum over~$\ell$ makes it challenging to use this formula in concrete cases.

\subsection{Examples and comparison with other methods}

We used our method in cases extracted from the literature of sieve theory.
In earlier works, the numerical estimation of sieve integrals was usually carried out in the following ways:
\begin{itemize}
  \item By dissecting the polytope along a~$k$-dimensional fine grid, and bounding the integrand in each box~\cite{Chen1973,Harman1983,Li2025,Merikoski2023},
  \item By using Monte-Carlo based methods; this seems to be the preferred method used by the Mathematica closed-source software~\cite{Mathematica}.
\end{itemize}
See~\cite[Chapter~5]{DavisRabinowitz1984} for a description of the state-of-the-art in both types of methods.

Both types of methods are typically badly behaved in terms of the precision. In particular, methods based on the Monte Carlo algorithm will reach~$n$ bits of precision in time~$O(2^{n/2})$. Their main advantage, however, is their polynomial dependence in terms of the dimension.
By contrast, the method described here behaves exponentially in the dimension~$k$, but polynomially in the precision.

All computations were made using four cores in parallel (AMD EPYC 7452 2.2GHz).

\subsubsection{Chen's work~\cite{Chen1973} on twin almost primes}\label{sec:chens-work}

Chen's celebrated work~\cite{Chen1973} showed the infinitude of primes~$p$ for which~$p+2$ is either prime or product of two primes, by proving a lower bound of the right order of magnitude for such primes.
The current best result in this direction is the recent work of R. Li~\cite{Li2026}.

We tested our algorithm on two sieve integrals arising in~\cite{Chen1973}.
The first one, in Lemma~8 \textit{ibid.}, is
$$ C_1 :=  I^*(P), \qquad P = \{\tfrac1{10} < t_1 < \tfrac13 < t_2 < t_3\} $$
where we recall the definition~\eqref{eq:def-I*}. We get $C_1 = 0.4909952010 \pm 1.43 \times 10^{-11}$ in about one second.
The computation time as a function of the requested number of bits, which is polynomial by Theorem~\ref{th:poly-complexity}, is represented below.
\begin{center}
  \begin{tikzpicture}
    \begin{axis}[
      xlabel={Requested precision (bits), average of three runs},
      ylabel={Time (s)},
      ymin=0,
      xmin=20, xmax=88,
      grid=both,
      width=12cm,
      height=5cm
      ]

      \addplot[
      only marks,
      mark=x,
      ] coordinates {
        (14, 0.23) (15, 0.23) (16, 0.25) (17, 0.25) (18, 0.26) (19, 0.26) (20, 0.27) (21, 0.29) (22, 0.3) (23, 0.3) (24, 0.32) (25, 0.34) (26, 0.37) (27, 0.38) (28, 0.39) (29, 0.43) (30, 0.52) (31, 0.53) (32, 0.58) (33, 0.61) (34, 0.69) (35, 0.71) (36, 0.76) (37, 0.81) (38, 0.84) (39, 0.9) (40, 1.03) (41, 1.07) (42, 1.19) (43, 1.34) (44, 1.4) (45, 1.66) (46, 1.85) (47, 1.81) (48, 2.11) (49, 2.31) (50, 2.35) (51, 2.71) (52, 2.99) (53, 3.14) (54, 3.53) (55, 3.89) (56, 4.04) (57, 4.53) (58, 5.08) (59, 5.42) (60, 6.02) (61, 6.39) (62, 6.93) (63, 7.72) (64, 8.14) (65, 8.88) (66, 9.87) (67, 10.46) (68, 11.88) (69, 12.42) (70, 13.2) (71, 14.9) (72, 15.49) (73, 17.09) (74, 18.69) (75, 19.24) (76, 21.03) (77, 23.11) (78, 23.9) (79, 26.26) (80, 28.46) (81, 29.37) (82, 32.31) (83, 34.79) (84, 37.5) (85, 39.69) (86, 42.58) (87, 46.16) (88, 48.16) (89, 51.89) (90, 56.08) (91, 58.6)
      };
    \end{axis}
  \end{tikzpicture}  
\end{center}

Note that this integral reduces to a single-dimensional integral of an elementary function~\cite[p.~326]{HalberstamRichert1974}, for which the double-exponential algorithm~\cite{TakahasiMori1974} is much faster, as can easily be checked \emph{e.g.} in PARI/GP~\cite{PARI-GP}.

The second computation is found on~\cite[p.~175]{Chen1973}, and reduces to
$$ C_2 := I(P_1) - \tfrac12 I(P_2) - \tfrac12 I(P_3), $$
where
\begin{align*}
  P_1 ={}& \{1 < t_3 < t_2-1 < t_1-2 < 2\}, \\
  P_2 ={}& \{1 < t_3 < t_2-1 < 3 - 10t_1 < 2\}, \\
  P_3 ={}& \{1 < t_3 < t_2-1 < 10t_1-2 < 2\}.
\end{align*}
We get~$C_2 = -0.0148863467 \pm 4.39\times 10^{-11}$ in about 4 seconds.

\subsubsection{Ford-Maynard's work~\cite{FordMaynard2024} on prime-producing sieves}

In Ford and Maynard's recent work~\cite{FordMaynard2024} on prime-producing sieves, several sieve integrals are computed in the context of the numerical estimation of criteria which determine whether certain type~I and type~II hypotheses on a sequence are enough to ensure the existence of primes in that sequence (see \emph{e.g.} Theorem~2.7 there). These integrals are called~$I_3, \dotsc, I_6$ in page~66 \textit{ibid.} We confirm rigorously their numerical computations up to the claimed digits, and for benchmark, we computed also~$I_5' = -3.77\times 10^{-8} \pm 5.14\times 10^{-11}$ and~$I_6 \in [-9.9996 \times 10^{-13}, -9.9994 \times 10^{-13}]$. We report below the requested precision and the computation time (in seconds, per integral).
\[
  \begin{array}{|l|c|c|c|c|c|}
    \hline
    \text{Integral} & I_3 & I_4 & I_5 & I_5' & I_6\\
    \hline
    \text{Requested precision (bits)} & 33 & 35 & 32 & 37 & 63 \\
    \hline
    \text{Number of non-empty integrals} & 2 & 3 & 6 & 4 & 1 \\
    \hline
    \text{Computation time per integral (s)} & 0.5 & 3.0 & 2.0 & 0.1 & 0.1 \\
    \hline
  \end{array}
\]

\subsubsection{Examples using the Harman sieve}\label{sec:exampl-using-harm}

Among the vast literature on the Harman's alternative sieve method, we chose to focus on the fairly recent works~\cite{Maynard2019}, \cite{Stadlmann2022} and~\cite{Merikoski2023}. All of these works make use of numerical computations to prove positivity of main terms arising from the use of Harman's sieve.

\bigskip{}

First we discuss the work of Maynard~\cite{Maynard2019} on primes with a missing digit in small numeration bases.
We recomputed for benchmark the integrals~$I_1, \dotsc, I_9$ from~\cite[p.~146]{Maynard2019}.
The integrals~$I_5$,~$I_6$ and~$I_9$ were the most demanding, which is expected from the large size of the argument in the Buchstab function, and the fact that the involved polytopes are not convex (which requires a preliminary splitting).
The overall computation time is about 50 seconds.
\[
  \begin{array}{|c|l|l|l|}
    \hline
    \text{Integral} & \text{Time (s)} & \text{Value} & \text{Upper-bound from~\cite{Maynard2019}} \\
    \hline
    I_1 & 0.3 & 0.02894233 \pm 2.8\times 10^{-9} & 0.02895 \\
    I_2 & 0.4 & 0.35717085 \pm 8.3\times 10^{-9} & 0.35718 \\
    I_3 & 0.4 & 0.0140114 \pm 7.2\times 10^{-8} & 0.01402 \\
    I_4 & 0.7 & 0.04237 \pm 4.3\times 10^{-6} & 0.04238 \\
    I_5 & 19.6 & 0.0552 \pm 6.5\times 10^{-5} & 0.05547 \\
    I_6 & 12.9 & 0.0659 \pm 8.6\times 10^{-5} & 0.06622 \\
    I_7 & 0.2 & 0.21878 \pm 5.9\times 10^{-6} & 0.21879 \\
    I_8 & 0.3 & 0.2023 \pm 3.7\times 10^{-5} & 0.20339 \\
    I_9 & 3.0 & 0.00918 \pm 7.6\times 10^{-6} & 0.00924 \\
    \hline
  \end{array}
\]
Next we compare the values of the integrals~$I_1$ and~$I_7$ computed by our method with those computed by the Mathematica software in~\cite{Maynard2019}, taking advantage of the fact that these values are recorded in the Mathematica code companion in the arXiv source files \emph{ibid}. Note that for these two values the argument of the Buchstab function is less than~$3$ and thus the computations from~\cite{Maynard2019} approximates the actual value of the integral, rather than an upper-bound.
\[
  \begin{array}{|c|l|c|}
    \hline
    \text{Integral} & \Bigl\{ \begin{array}{l} \text{value returned by Mathematica} \\ \text{value computed by the present method} \end{array}  & \text{difference} \\
    \hline
    I_1 & \Bigl\{\begin{array}{l} 0.02894232875|40233548637864210721 \\ 0.02894232875|3689385 \pm 3.5\times 10^{-19} \end{array} & 4 \times 10^{-13} \\
    I_7 & \Bigl\{\begin{array}{l} 0.218781293515|920652506327212130 \\ 0.218781293515|89336 \pm 4 \times 10^{-18} \end{array} & 3 \times 10^{-14} \\
    \hline
  \end{array}
\]
We indicated the first digit where the results differ. It would be interesting to compare the actual precision of the Mathematica results with the error estimate given by Mathematica.

\bigskip{}

We discuss next the work of Stadlmann~\cite{Stadlmann2022}, which studies the mean square gap between consecutive primes, using an involved analysis with the Harman sieve. We recomputed, for benchmark, fifty-seven of the various integrals in~\cite{Stadlmann2022}, in the same computational conditions as above. The total computation time was 2 minutes, on average about 2 seconds per integral.
Additionally, we estimated numerically the two $6$-dimensional integrals on \cite[p.~67]{Stadlmann2022}. They are listed below, along with the computational time.
\[
  \begin{array}{|l|l|l|l|l|}
    \hline
    \begin{array}{l} \text{Location} \\ \text{in~\cite{Stadlmann2022}} \end{array} &
    \text{Time (s)} &
    \text{Value} &
    \text{Bound from~\cite{Stadlmann2022}} \\
    \hline
    \text{p.~58} & 15.2 & 0.27674 \pm 1.3\times 10^{-4} & < 0.296 \\
    \text{p.~60} & 0.2 & 0.18626 \pm 3.6\times 10^{-6} & < 0.1993 \\
    \text{p.~65} & 0.2 & 0.15737 \pm 1.9\times 10^{-6} & < 0.1723 \\
    \text{p.~67} & 41.9 & 0.003 \pm 4.2 \times 10^{-4} & < 0.056 \\
    \text{p.~67} & 52.7 & 0.005 \pm 5.3 \times 10^{-4} & < 0.035 \\
    \text{p.~68} & 0.3 & 0.27893 \pm 5.3\times 10^{-6} & < 0.302 \\
    \text{p.~69} & 0.2 & 0.16554 \pm 4.3\times 10^{-6} & < 0.1769 \\
    \text{p.~69} & 0.2 & 0.11649 \pm 2.15\times 10^{-6} & < 0.1266 \\
    \text{p.~79} & 0.2 & 0.19179 \pm 2.41\times 10^{-6} & < 0.2102 \\
    \hline
  \end{array}
\]

\bigskip{}

Finally, we discuss the works of Merikoski and Grimmelt-Merikoski~\cite{Merikoski2023,GrimmeltMerikoski2025}, improved by Runbo Li~\cite{Li2024}, on the largest prime divisors of quadratic polynomials. We tested our algorithm to compute the integrals~$F_1, \dotsc, F_6$ from~\cite[Section~2.6]{Merikoski2023}, which were computed with a Python script in~\cite{Merikoski2023} and using Mathematica in~\cite{Li2024}. We requested~30 bits of precision in order to confirm the results from~\cite{Li2024}.
\[
  \begin{array}{|c|l|l|l|l|}
    \hline
    \text{Integral} & \text{Time (s)} & \text{Value} & \text{Bound from~\cite{Merikoski2023}} & \text{Bound from~\cite{Li2024}} \\
    \hline
    F_1 & 0.2 & 0.02861092 \pm 3.8\times 10^{-9} & <0.0287 & < 0.028611\\
    F_2 & 0.3 & 0.08606199 \pm 4.8\times 10^{-9}  & <0.08622 & < 0.086062 \\
    F_3 & 0.4 & 0.03099115 \pm 8.8\times 10^{-9} & <0.03107 & < 0.30992 \\
    F_4 & 0.6 & 6.095\times 10^{-5} \pm 7.8\times 10^{-9} & <1.1 \times 10^{-4} & < 1 \times 10^{-4}\\
    F_6 & 0.3 & 0.05986241 \pm 4.8 \times 10^{-9} & > 0.035631 & > 0.059841 \\
    \hline
  \end{array}
\]
We stress the small computation time (about 1.5 second overall).

\begin{remark}
  Our method is insensitive to the presence of a small-degree polynomial in the integrand: this allows to implement rather easily Richert's weighted sieve method~\cite[Chapter~25.3]{FriedlanderIwaniec2010}. See~\cite{MatomaekiZuniga-Alterman2025} for a recent work on this topic.
\end{remark}

\subsection{Source code}

The source code for our implementation of the method is publicly available on the GitHub repository~\cite{github-sieve-integral} under GNU General Public License 3. This repository also contains the notebooks pertaining to the computations described in Sections~\ref{sec:chens-work}--\ref{sec:exampl-using-harm}.

\section{Density of integers with a localized divisor}\label{sec:hab}

We apply our algorithm to the computation of the density
$$ h(\alpha, \beta) := \lim_{x\to\infty} \frac1x \card\{n\leq x, \exists d\mid n, x^\alpha < d < x^\beta\}. $$
This question is related to the celebrated multiplication problem of Erdős~\cite{Erdos1960}. In~\cite{Tenenbaum1980}, Tenenbaum studies this quantity and provides several bounds, and an explicit formula in some cases. One of the main observations in~\cite{Tenenbaum1980} is that the existence of a divisor~$d\mid n$ with~$x^\alpha < d < x^\beta$ depends essentially only on the large prime divisors of~$d$. In a recent work, Haddad obtained an explicit formula which, in particular, makes effective this computation: in~\cite[Theorem~2]{Haddad2026} it is proved that
\begin{equation}\label{eq:haddad-hab}
  h(\alpha, \beta) = 1 - \sum_P \int_P \rho\Big(\frac{1 - \sum_j t_j}{t_k}\Big) \frac{\df t}{t_1 \dotsb t_k}
\end{equation}
where~$\rho$ is the Dickman function, $k = \lfloor 1/(\beta - \alpha)\rfloor$ and the sum over~$P$ runs over a family of at most~$2^{2^k}$ polytopes in~$\R_{>0}^k$.

In this section, we provide another formula for~$h(\alpha, \beta)$, which is closely related to the above but comprises a reasonable number of terms, and can therefore be used for actual numerical computations. To state our results, it will be convenient to use the notation~$B(u)$ defined in~\eqref{eq:def-Omega}. We also recall the definition of the Dickman function~$\rho$ in~\eqref{eq:def-rho}.

\begin{theorem}\label{th:expr-h}
  For any~$0<\alpha<\beta<1$, there are families~$\cP_k$, $\cP'_k$ of convex polytopes in~$\R^k$ ($k\geq 2$) such that
  \begin{equation}\label{eq:formule_h}
    \begin{aligned}
      h(\alpha, \beta) ={}& 1 - B\Big(\frac1\beta + 1\Big) + B\Big(\frac{1-\alpha}\beta + 1\Big) \\
      & {} - \sum_{2\leq k < 1/(\beta-\alpha)} \Big( \sum_{P \in \cP_k} \int_P \rho\Big(\frac{1-\sum_j t_j - s}{\beta-\alpha}\Big) B\Big(\frac s\beta\Big) \frac{\df t \df s}{t_1 \dotsb t_k s} \\
      & \qquad \qquad \qquad + \sum_{P \in \cP'_k} \int_P \rho\Big(\frac{1-\sum_j t_j}{\beta-\alpha}\Big) \frac{\df t}{t_1 \dotsb t_k} \Big).
    \end{aligned}
  \end{equation}
  Moreover we have~$\card(\cP_k), \card(\cP'_k) \leq n_k$ where~$n_k$ is given in~\eqref{eq:nk}. The elements of~$\cP_k$ and~$\cP_k'$ are defined in~\eqref{eq:def-DeltaF} and \eqref{eq:def-DeltaFprime} below.
  
  Finally, for any~$k_0\geq 3$, the contribution of those~$k\geq k_0$ in the above sum is at most
  \begin{align}
    & \abs{\sum_{k\geq k_0}} \leq \min(E_1(k_0), E_2(k_0)), \quad \text{ with} \label{eq:maj_k_grand} \\
    & E_1(k_0) = \frac{\max\{0, \frac1{\beta-\alpha} - k_0\}^{k_0}}{(k_0!)^2}, \notag \\
    & E_2(k_0) = \e^\gamma \lambda \sum_{k\geq k_0} \frac{\log(1/\lambda)^k}{k!}, && (\lambda = \tfrac{(1-\alpha)(\beta-\alpha)}{\alpha}) \notag
  \end{align}
\end{theorem}

The value~$n_7 = 43141$ should be compared with~$2^{2^7} \approx 10^{38}$. The formula~\eqref{eq:formule_h} makes the computation of~$h(\alpha, \beta)$ possible on mainstream laptop computers, at least as long as~$\beta-\alpha$ is not too close to~$0$; our numerical observations indicate a sharp increase in computational difficulty as~$\beta-\alpha$ approaches roughly~$1/15$.

\subsection{Monotone boolean functions for the dominant order}

One of the key steps in Theorem~\ref{th:expr-h} involves a combinatorial expansion of a characteristic function of the type
\begin{equation}
  \begin{aligned}
    &\1\Big( (t_1, \dots t_k) \text{ has no subsum in } [\alpha, \beta]\Big) \\
    ={}& \prod_{J\subset \{1, \dotsc, k\}} \Big(\1\big(\sum_{j\in J}t_j < \alpha\big) + \1\big(\sum_{j\in J}t_j > \beta\big)\Big).
  \end{aligned}\label{eq:subsum-prod-charfun}
\end{equation}
Expanding this product leads to a sum of~$2^{2^k}$ characteristic functions of the shape~$\1( t \in P)$ where~$t = (t_1, \dotsc, t_k)$ and~$P$ is a convex polytope, intersection of hyperplanes of the form~$\sum_{j\in J} t_j < \alpha$ or~$\sum_{j\in J}t_j > \beta$. This is essentially the reason for the bound~$2^{2^k}$ for the number of terms in the expression~\eqref{eq:haddad-hab} obtained in~\cite{Haddad2026}.

Our main observation is that many of these polytopes are actually empty, and in order to describe them in a practical way, it is convenient to introduce boolean functions.

\begin{definition}
  \begin{enumerate}[label=(\roman*)]
    \item Denote~$V_k := \{0, 1\}^k$ the boolean vectors on~$k$ components.
    \item Denote~$\cB_k = \{ F : \{0, 1\}^k \to \{0, 1\} \}$ the set of boolean functions on~$V_k$.
    \item The dominant order~$\leqslant$ on~$V_k$ is defined by
    $$ \forall x, y \in V_k, \qquad x\leqslant y \quad \iff \quad \forall \ell\leq k, \sum_{1\leq i \leq \ell} x_i \leq \sum_{1\leq i \leq \ell} y_i. $$
    \item We let~$\cF_k$ to be those boolean functions which are increasing for the dominant order.
  \end{enumerate}
\end{definition}

The elements of~$\cF_k$ are sometimes referred to as \emph{regular} boolean functions~\cite[ex.~110, p.~93]{Knuth2011}.

While~$\card \cB_k = 2^{2^k}$, the cardinality of~$\cF_k$ is much smaller. We record the values for~$k\leq 7$ below; this is sequence~A132183 in the Online Encyclopedia of Integer Sequences~\cite{OEIS}.
\begin{center}
  \begin{tabular}{|c|cccccccc|}
    \hline 
    $k$ & $0$ & $1$ &  $2$ &  $3$ &  $4$ &  $5$ &  $6$ &  $7$ \\
    \hline
    $\card(\cF_k)$ & $2$ & $3$ & $5$ & $10$ & $27$ & $119$ & $1173$ & $44315$ \\
    \hline
  \end{tabular}
\end{center}
The following key proposition relates the set~$\cF_k$ to the characteristic function mentioned above.

\begin{proposition}\label{prop:char-fun}
  Let~$\alpha < \beta$ and~$t \in \R^k$. Assume that~$t_1 > \dotsb > t_k> 0$. Then we have
  $$ \1\Big((t_1, \dots t_k) \text{ has no subsum in } [\alpha, \beta]\Big) = \sum_{F\in \cF_k} \1_{P_F}(t), $$
  where~$P_F=P_F(\alpha,\beta) \subset \R_{>0}^k$ is defined by
  $$ P_F(\alpha,\beta) = \Big\{t\in \R^k, t_1 > \dotsb > t_k>0, \forall x\in V_k, \begin{cases} F(x)=1\implies \sum_j x_j t_j > \beta \\ F(x)=0 \implies \sum_j x_j t_j < \alpha\end{cases} \Big\}. $$
  In particular, $P_F(\alpha,\beta)$ is empty whenever $\alpha\leq 0$
\end{proposition}

\begin{proof}
  We start from formula~\eqref{eq:subsum-prod-charfun} and expand the product on the right-hand side. We order the resulting product by functions~$g:\cP(\{1, \dotsc, k\}) \to \{0, 1\}$ on subsets of~$\{1, \dotsc, k\}$, where for any subset~$J$, we have~$g(J) = 0$ or~$1$ according to whether we select the condition~$\sum_{j\in J} t_j < \alpha$ or~$\sum_{j\in J}t_j > \beta$. We get as a result
  \begin{align*}
    & \1\Big((t_1, \dots t_k) \text{ has no subsum in } [\alpha, \beta]\Big) \\
    &{} = \sum_g \1\Big(\forall J\subset \{1, \dotsc, k\}, \begin{cases} g(J) = 1 \implies \sum_{j\in J} t_j > \beta, \\ g(J) = 0 \implies \sum_{j\in J} t_j < \alpha. \end{cases}\Big).    
  \end{align*}
  We associate bijectively to every subset~$J \subset\{1, \dotsc, k\}$, an element~$x\in V_k$, by the relation~$x_i = \1(i \in J)$. Precomposing by this bijection, the maps~$g:\cP(\{1, \dotsc, k\}) \to \{0, 1\}$ correspond to function in~$\cB_k$, and therefore the expression above is
  $$ \sum_{F\in \cB_k} \1\Big(\forall x\in V_k, \begin{cases} F(x) = 1 \implies \sum_j x_j t_j > \beta \\ F(x) = 0 \implies \sum_j x_j t_j < \alpha \end{cases}\Big). $$
  Recall that~$t_1 > \dotsc > t_k > 0$ by assumption. It remains to justify that for each~$F\in \cB_k$, the summand is empty unless~$F\in \cF_k$. Let~$F$ be a function for which the indicator function is~$1$.
  Suppose~$x, y \in V_k$ are such that~$x \leq y$ for the dominant order, and let~$\delta_\ell := \sum_{1\leq i \leq \ell} (y_i - x_i)$ so that~$\delta_\ell \geq 0$ for all~$\ell$ by hypothesis. Then by partial summation
  $$ \sum_j (y_j - x_j) t_j = \sum_\ell \delta_\ell (t_{\ell-1} - t_\ell) \geq 0. $$
  Therefore~$\sum_j x_j t_j \leq \sum_j y_j t_j$. This excludes the possibility that~$F(x) = 1$ and~$F(y) = 0$ since otherwise we would have~$\sum_j x_j t_j > \beta \geq \alpha > \sum_j y_j t_j$. Therefore~$F(x) \leq F(y)$. This proves that~$F$ is monotonic for~$\leq$ as claimed.
\end{proof}

\subsection{Effective reduction to conditions on large primes}

In this section, we complete the proof of Theorem~\ref{th:expr-h}.
Given~$n\leq x$ and real numbers~$\alpha \leq \beta$, we denote
$$ \deltabar(n; [\alpha, \beta]) := \1(\forall d\mid n, d\not\in [x^\alpha, x^\beta]). $$
We omit the dependency on~$x$ in the notation.
Note that by definition
$$ h(\alpha, \beta) = 1 - \lim_{x\to\infty} \frac1x \sum_{n\leq x} \deltabar(n; [\alpha, \beta]). $$
Our key lemma is the following.
\begin{lemma}\label{lem:jump}
  Let~$n\geq 1$, and decompose~$n = n_1 n_2 n_3$ with
  \begin{equation}
    n_1 = \prod_{\substack{p^\nu \| n \\ p \leq x^{\beta-\alpha}}} p^\nu, \qquad n_2 = \prod_{\substack{p^\nu \| n \\ x^{\beta-\alpha} < p \leq x^\beta}} p^\nu, \qquad n_3 = \prod_{\substack{p^\nu \| n \\ x^\beta < p}} p^\nu.\label{eq:n-decomp-n1n2n3}
  \end{equation}
  Then we have
  $$ \deltabar(n; [\alpha, \beta]) = \deltabar(n_2; [\alpha-\tfrac{\log n_1}{\log x}, \beta]). $$
\end{lemma}
The useful feature here is the fact that~$n_2$ consists only of large prime divisors, which as was observed in Tenenbaum's work~\cite[p.~31]{Tenenbaum1980}, is key to reduce the computation of~$h(\alpha, \beta)$ to a finite sum. A similar phenomenon takes place in the work of Haddad~\cite{Haddad2026} (see equation~(23) there).

\begin{proof}
  First it is obvious that~$\deltabar(n; [\alpha, \beta]) = \deltabar(n_1n_2; [\alpha, \beta])$, since any divisor~$d\mid n$ sharing a prime divisor~$p$ with~$n_3$ must be at least as large as~$p$, but~$p>x^\beta$ by definition of~$n_3$.

  Assume that there exists~$d\mid n$, $x^\alpha \leq d \leq x^\beta$, so that~$\deltabar(n; [\alpha, \beta])=0$, and split~$d = d_1 d_2$ with~$d_1 \mid n_1$ and~$d_2 \mid n_2$. Then we have obviously~$d_2 \leq d \leq x^\alpha$, and also~$d_2 \geq d / n_1 \geq x^\alpha / n_1$. Therefore~$\deltabar(n_2; [\alpha-\frac{\log n_1}{\log x}, \beta]) = 0$ also.

  Assume conversely that there exists~$d_2 \mid n_2$, $x^\alpha / n_1 \leq d_2 \leq x^\beta$. If~$d_2 \geq x^\alpha$, we are done. Otherwise write~$n_1 = p_1 \dots p_r$ with~$p_j$ primes, not necessarily distinct, and consider the divisors
  $$ d_2, \quad p_1 d_2, \quad p_1p_2 d_2, \quad \dotsc, \quad p_1 \dotsb p_r d_2 = n_1d_2. $$
  By hypothesis we have~$d_2 < x^\alpha$ and~$n_1 d_2 \geq x^\alpha$. Moreover, since~$p_j \leq x^{\beta - \alpha}$ by definition of~$n_1$, the ratio between two consecutive numbers in this sequence is at most~$x^{\beta - \alpha}$. Therefore at least one such element must belong to~$[x^\alpha, x^\beta]$, and so~$\deltabar(n; [\alpha, \beta])=0$.
\end{proof}

\begin{proof}[Proof of Theorem~\ref{th:expr-h}]
  We compute the limit
  $$ 1 - h(\alpha, \beta) = \lim_{x\to \infty} \frac1x \sum_{n\leq x} \deltabar(n; [\alpha, \beta]) $$
  using the decomposition from Lemma~\ref{lem:jump}. First we isolate the contribution from~$n_2 \leq x^\alpha / n_1$, which is
  \begin{align*}
    S_1 ={}& \frac1x \ssum{n_1 \leq x \\ P^+(n_1) \leq x^{\beta-\alpha}} \ssum{n_2 \leq x^\alpha / n_1 \\ p\mid n_2 \implies x^{\beta - \alpha} < p \leq x^\beta} \ssum{n_3 \leq x/n_1n_2 \\ P^-(n_3) > x^\beta} \deltabar(n_2; [\alpha - \tfrac{\log n_1}{\log x}, \beta]) \\
    ={}& \frac1x \ssum{m < x^\alpha \\ P^+(m) \leq x^{\beta}} \ssum{n_3 \leq x/m \\ P^-(n_3) > x^\beta} 1 \\ 
    ={}& \ssum{m < x^\alpha} \frac1{m \log(x/m)} B\Big(\frac{1 - \frac{\log m}{\log x}}{\beta}\Big) + O\Big(\frac1{\log x}\Big)
  \end{align*}
  by \cite[Theorem~III.6.3]{Tenenbaum2015}. The resulting sum is evaluated, using partial summation, as
  \begin{align*}
    S_1 ={}& O\Big(\frac1{\log x}\Big) + \int_1^{x^\alpha} \frac1{z\log(x/z)} B\Big(\frac{1 - \frac{\log z}{\log x}}\beta\Big)\df z \\
    ={}& O\Big(\frac1{\log x}\Big) + \int_{1-\alpha}^1 B\Big(\frac s\beta\Big) \frac{\df s}{s}
  \end{align*}
  by a change of variables~$z = x^{1-s}$. Writing~$B(s/\beta)/s = B'(s/\beta+1)/\beta$, we conclude that
  $$ S_1 = B\Big(\frac1\beta+1\Big) - B\Big(\frac{1-\alpha}\beta + 1\Big) + O\Big(\frac1{\log x}\Big) $$
  which provides the first term in~\eqref{eq:formule_h}.

  Next we are to evaluate
  \begin{equation}\label{eq:n1n2_grands}
    \ssum{n = n_1 n_2 n_3 \leq x \\ n_1n_2 > x^\alpha} \deltabar(n_2; [\alpha - \tfrac{\log n_1}{\log x}, \beta]),
  \end{equation}
  where the decomposition of~$n$ is according to~\eqref{eq:n-decomp-n1n2n3}. We focus on the contribution of~$n_3>1$, which we call~$S_2$.
  We start by observing that the contribution to~$S_2$ of those integers~$n_2$ divisible by the square of a prime, is at most
  $$ \frac1x \ssum{n\leq x \\ \exists p>x^{\beta-\alpha}, p^2 \mid n} 1 \ll x^{\alpha-\beta} = o(1). $$
  Next, we write a squarefree~$n_2 > 1$ with~$P^-(n_2) > x^{\beta - \alpha}$ in a unique way as~$n_2 = p_1 \dotsc p_k$, where~$1\leq k \leq 1/(\beta-\alpha)$ and~$p_j$ are primes subject to~$p_1 > \dotsb > p_k$. We write
  $$ t_j = \frac{\log p_j}{\log x}, \qquad t = (t_1, \dotsc, t_k), $$
  and remark that, by definition of~$\deltabar$ and Proposition~\ref{prop:char-fun}, we have
  \begin{align*}
    \deltabar(n_2; [\alpha - \tfrac{\log n_1}{\log x}, \beta]) ={}& \1( (t_j) \text{ has no subsum in } [\alpha - \tfrac{\log n_1}{\log x}, \beta]) \\
    ={}& \sum_{F\in \cF_k} \1(t \in P_F(\alpha - \tfrac{\log n_1}{\log x}, \beta)).
  \end{align*}
  The quantity above is $0$ whenever $n_1>x^\alpha$. We deduce
  \begin{equation}\label{eq:S_2}
    S_2 = o(1) + \sum_k \sum_{F\in \cF_k} \frac1x \ssum{n_1 \leq x^\alpha \\ P^+(n_1) \leq x^{\beta - \alpha}} \ssum{1 < n_3 \leq x/n_1 \\ P^-(n_3) > x^\beta} \ssum{x^\beta\geq p_1 > \dotsb > p_k > x^{\beta-\alpha}\\ x^\alpha/n_1 < p_1 \dotsb p_k \leq x/(n_1n_3)} \1(t \in P_F(\alpha - \tfrac{\log n_1}{\log x}, \beta)). 
  \end{equation}
  We fix~$k, F$ and evaluate the inner sums, which we call~$S_2(F)$.
  Note that the summand does not depend on~$n_3$. Executing first the sum over~$n_3$ therefore yields
  $$ \frac1x \ssum{1<n_3 \leq x/(n_1 p_1 \dotsb p_k) \\ P^-(n_3) > x^\beta} 1 = \frac{1}{\beta\log x} \frac{\omega(\frac{1 - \sum_j t_j - \frac{\log n_1}{\log x}}{\beta})}{n_1 p_1 \dotsb p_k} + O\Big(\frac1{n_1 p_1 \dotsb p_k (\log x)^2}\Big). $$
  Here it is important to note that the contribution of~$n_3=1$ was removed. The error term contributes trivially~$O(1/\log x)$ to~$S_2$, and we get
  \begin{align}\label{eq:S_2(F)}
    S_2(F) &= o(1) + \frac{1}{\beta\log x}\ssum{n_1 \leq x^\alpha \\ P^+(n_1) \leq x^{\beta - \alpha}} \frac{1}{n_1}G_k\Big(\frac{\log n_1}{\log x}\Big)
  \end{align}
  where we have set for $\nu\in[0,\alpha]$,
  $$ L_k(\nu) = \ssum{p_1, \dotsc, p_k \\ t\in P_F(\alpha,\beta,\nu)}\frac{\omega(\frac{1 - \nu-\sum_j t_j}{\beta})}{p_1 \dotsb p_k}, $$
  where $P_F(\alpha,\beta,\nu)$ is the polytope defined by
  \[
    P_F(\alpha,\beta,\nu) := P_F(\alpha-\nu,\beta) \cap \left\{
      t\in\R^k, \quad
      \begin{aligned}
        & \beta-\alpha < t_k < \dotsb < t_1 < \beta,\\
        & \alpha-\nu < t_1+ \dotsb + t_k < 1-\nu
      \end{aligned}
    \right\}
  \]
  Here we see that if $F$ is the null boolean function, then the set $P_F(\alpha,\beta,\nu)$ is empty. Indeed, if $F(1,\ldots,1) = 0$, an element $t\in P_F(\alpha,\beta,\nu)$ would verify both $t_1+\cdots+t_k<\alpha-\nu$ and $t_1+\cdots+t_k>\alpha-\nu$ . Similarly, the set $P_F(\alpha,\beta,\nu)$ is empty if~$F(x)=1$ for some tuple~$x\in V_k$ having a single component equal to~$1$. Indeed, in such a case an element $t\in P_F(\alpha,\beta,\nu)$ would have a component $t_i$ verifying both $t_i<\beta$ and $t_i>\beta$. These observations lead to the fact that the $3$ boolean functions in $\cF_1$ do not contribute to~$S_2$, hence we can start the sum at $k=2$ in~\eqref{eq:S_2}.
  We are now in a position to express $L_k(\nu)$ in terms of integrals, using~\cite[Lemma~2.6]{Mounier2025}:
  $$L_k(\nu) = \int_{P_F(\alpha,\beta,\nu)}\omega\Big(\frac{1 - \nu-\sum_j t_j}{\beta}\Big)\cdot\frac{\mathrm{d}t}{t_1\cdots t_k} + O_k((\log x)^{-1})$$
  Inserting this expression in \eqref{eq:S_2(F)} and swapping the sum over $n_1$ and the integral yields
  \begin{equation}\label{eq:S_2(F)_2}
    S_2(F) = o(1) + \frac{1}{\beta\log x} \int_{\substack{\beta-\alpha < t_k < \ldots < t_1 < \beta \\ t_1 + \dotsb + t_k < 1 \\ t\in P_F(\alpha,\beta)}} \frac{\df t}{t_1\dotsb t_k}\ssum{x^A<n\leq x^B \\ P^+(n)\leq x^C} \frac 1n \omega\Big(\frac{1 - \frac{\log n}{\log x}-\sum_j t_j}{\beta}\Big)
  \end{equation}
  where~$ A = \max(0,\alpha-\sum_jt_j)$, $B = \min(1-\sum_jt_j, \alpha-\max_{x\in F^{-1}(0)}\sum_jx_jt_j)$, and~$C = \beta-\alpha$.
  We evaluate the inner sum using Dickman's theorem~\cite[Theorem~III.5.6]{Tenenbaum2015} and partial summation, as
  $$ \ssum{x^A < n\leq x^B \\ P^+(n)\leq x^C} \frac 1n \omega\Big(\frac{1 - \frac{\log n}{\log x} - \sum_j t_j}{\beta}\Big) = (\log x) \int_A^B \rho\big(\tfrac{v}{\beta-\alpha}\big) \omega\big(\tfrac{1 - v - \sum_j t_j}{\beta}\big)\df v + O_{\beta-\alpha}(1) $$
  uniformly in~$A$ and~$B$.Inserting this expression in \eqref{eq:S_2(F)_2} yields
  \begin{align*}
    S_2(F) &= o(1) + \frac{1}{\beta} \int_0^\alpha \rho\Big(\frac{v}{\beta -\alpha}\Big) \int_{P_F(\alpha,\beta,v)} \frac{1}{t_1\cdots t_k} \omega\Big(\frac{1 - v - \sum_j t_j}{\beta}\Big) \df t \df s \\
    &= o(1) - \int_{P_F} \rho\Big(\frac{1-\sum_j t_j - s}{\beta-\alpha}\Big) B\Big(\frac s\beta\Big) \frac{\df t \df s}{t_1 \dotsb t_k s}
  \end{align*}
  where now
  \begin{equation}
    P_F = P_F(\alpha, \beta) = \begin{cases}
      \beta > t_1 > \dotsb > t_k > \beta - \alpha, \\
      0<s<1-\alpha<s + \sum_i t_i < 1 \\
      F(J) = 1 \implies \sum_J t_j > \beta, \\
      F(J) = 0 \implies \sum_{J^c} t_j > 1-\alpha-s.
    \end{cases}\label{eq:def-DeltaF}
  \end{equation}
  Inserting this in \eqref{eq:S_2} yields
  $$S_2 = o(1)-\sum_{2\leq k < 1/(\beta-\alpha)}  \sum_{P \in \cP_k} \int_P \rho\Big(\frac{1-\sum_j t_j - s}{\beta-\alpha}\Big) B\Big(\frac s\beta\Big) \frac{\df t \df s}{t_1 \dotsb t_k s}$$
  where~$\cP_k = \cP_k(\alpha, \beta) = \{P_F, F\in\cF_k\}\setminus\{\varnothing\}$ (we do not take into account those~$F$ for which~$P_F$ is empty).
  It remains to estimate the contribution of $n_3=1$ in \eqref{eq:n1n2_grands}, which we call $S_3$. In this case, the sum over $n_3$ disappears, so does the $B$ function. Following the same steps as previously, we get
  \begin{equation}\label{eq:S_3}
    S_3 = o(1) + \sum_{2\leq k < 1/(\beta-\alpha)}\sum_{P \in \cP'_k} \int_P \rho\Big(\frac{1-\sum_j t_j}{\beta-\alpha}\Big) \frac{\df t}{t_1 \dotsb t_k}
  \end{equation}
  where $\cP_k' = \{P_F', F\in\cF_k\}\setminus\{\varnothing\}$ and
  \begin{equation}
    P_F' = \begin{cases}
      \beta > t_1 > \dotsb > t_k > \beta - \alpha, \\
      1-\alpha< \sum_i t_i < 1 \\
      F(J) = 1 \implies \sum_J t_j > \beta, \\
      F(J) = 0 \implies \sum_{J^c} t_j > 1-\alpha.
    \end{cases}\label{eq:def-DeltaFprime}
  \end{equation}
  Finally, \eqref{eq:formule_h} follows from summing $S_1, S_2$ and $S_3$.

  Next, for each~$P\in \cP_k$, we note that the constraints in~$P$ are linear in the variables~$\alpha, \beta$. Therefore the set
  $$ \Pi_F := \{ (\alpha, \beta, t), 0<\alpha<\beta<1, t\in P_F(\alpha, \beta)\} $$
  is itself a polytope. Using the Polyhedron class of SageMath~\cite{SageMath}, we have computed for~$k\leq 7$ the complete list of polytopes~$\Pi_F$ for~$F \in\cF_k$, and we found that~$n_k$ of them are non-empty, where~$n_k$ are the numbers given in~\eqref{eq:nk}. This proves that~$\abs{\cP_k} \leq n_k$ uniformly in~$\alpha, \beta$. We carried out the same computations with~$\cP_k'$ and found that the same holds for the same values~$n_k$.

  It remains to prove the upper bound \eqref{eq:maj_k_grand}. First, note that $\cP_k = \cP_{k'} = \varnothing$ whenever $k\geq \frac{1}{\beta-\alpha}$. Let $k_0 < \frac{1}{\beta-\alpha}$.
  For all~$n\in\N$ and~$y\geq 1$, define~$\omega_y(n) := \sum_{p\mid n,  p>y} 1$.
  For all~$k\geq 0$ and~$u>0$, it is proved by Knuth-Pardo~\cite{KnuthPardo1977} that the limit
  \begin{equation}
    G_k(u) := \lim_{x\to\infty} \frac1x \sum_{n\leq x} \1(\omega_{x^{1/(k+u)}}(n)\geq k)\label{eq:def-Gk}
  \end{equation}
  exists and defines a smooth function of~$u>0$: indeed the sum on the right-hand side corresponds to~$P_k(\frac1{k+u}, x)$ in the notation of~\cite[Section~3]{KnuthPardo1977}, and therefore~$G_k(u) = 1 - \rho_k(k + u)$ with the notations~\cite[eqs.~(4.2)--(4.4)]{KnuthPardo1977}.
  \begin{lemma}\label{lem:Gk}
    We have~$G_k(u) \leq u^k / (k!)^2$ for all~$u>0$ and~$k\geq 1$.
  \end{lemma}
  \begin{proof}
    The claim is obvious for~$k=0$ since then~$G_0(u) = 1$ for all~$u>0$. Let~$k\geq 1$ and suppose we have proved our claim for~$k-1$.
    By~\cite[eq.~(4.2)]{KnuthPardo1977}, we have
    $$ G_k(u) = \int_1^{k+u} (G_{k-1}(t-k) - G_k(t-k-1))\frac{\df t}{t} \leq \int_1^{k+u} G_{k-1}(t-k)\frac{\df t}{t} $$
    by positivity of~$G_k$. Since~$G_{k-1}(s)$ vanishes if~$s<0$, we may restrict integration to~$t\geq k$. Writing~$t = k+v$, we get
    \[ G_k(u) \leq \int_0^u G_{k-1}(v) \frac{\df v}{v + k}. \]
    Our claimed bound follows by induction as
    $$ G_k(u) \leq \frac1k \int_0^u G_{k-1}(v) \df v \leq \frac1{k (k-1)!^2} \int_0^u v^{k-1} \df v \leq \frac{u^k}{(k!)^2}. $$
  \end{proof}

  To prove the bound~$E_1$ in~\eqref{eq:maj_k_grand}, recall that~$k$ in our argument refers to the number of prime divisors of~$n$ which belong to the interval~$(x^{\beta-\alpha}, x^\beta]$. In particular, every such prime divisor is larger than~$x^{\beta-\alpha}$, and thus
  $$ \abs{\sum_{k\geq k_0} \ssum{n\leq x \\ \omega(n_2) = k} \deltabar(n_2; [\alpha - \tfrac{\log n_1}{\log x}, \beta]) } \leq \ssum{n\leq x \\ \omega_{x^{\beta-\alpha}}(n) \geq k_0} 1. $$
  This last quantity is zero if~$\beta - \alpha > 1/k_0$, and otherwise by Lemma~\ref{lem:Gk} it is equal to
  $$ G_k(1/(\beta-\alpha) - k_0) x  + o(x) \leq E_1(k_0) x + o(x) $$
  which proves our first bound in~\eqref{eq:maj_k_grand}.

  To prove the second bound~$E_2(k_0)$ of~\eqref{eq:maj_k_grand}, we upper-bound
  $$ \abs{\sum_{k\geq k_0}\ssum{n\leq x \\ \omega(n_2) = k} \deltabar(n_2; [\alpha - \tfrac{\log n_1}{\log x}, \beta]) } \leq \sum_{k\geq k_0} \psi_k(x, x^{\beta-\alpha}, x^\alpha) \leq x f_k^*\big(\tfrac{\beta-\alpha}{\beta}, \tfrac1\beta\big) + o(x) $$
  in the notations of~\cite[Théorème~C]{Tenenbaum1980}. Indeed, since any prime divisor~$p\mid n_2$ is at most~$x^\beta$, we must have~$p<x^\alpha$ since otherwise~$\deltabar(n_2; \dotsb) = 0$. Then our claimed bound follows by~\cite[Théorème~C.(iv)]{Tenenbaum1980}.

\end{proof}

\subsection{Numerical estimates on~$h(\alpha, \beta)$}\label{sec:numer-estim-hab}

The graph representing~$h(\alpha, \beta)$ in a region slightly away from the diagonal was depicted in Figure~\ref{fig:h-values} above.
The computational complexity grows very fast as~$(\alpha, \beta)$ approaches the diagonal. In Figure~\ref{fig:h-nb-integrals}, we depicted~$\log N$ where~$N$ is the number of integrals involved in the computation of a value~$h(\alpha, \beta)$; the maximum was~$N=6215$.

\begin{figure}[h]
  \caption{}
  \centering
  \begin{subfigure}[t]{.33\textwidth}
    \centering
    \image{1}{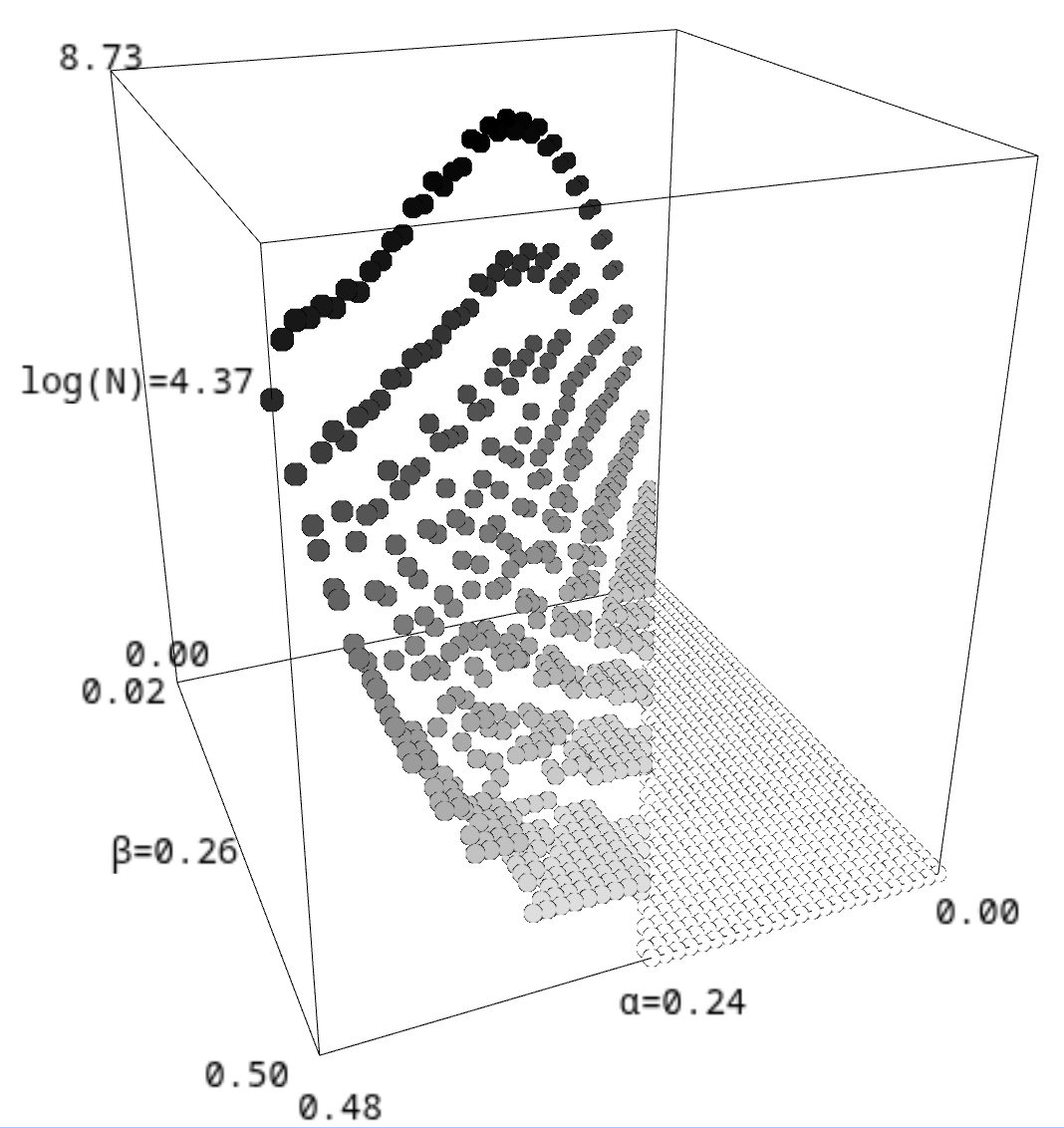}
    \caption{Log of the number of integrals involved in the computation of~$h(\alpha, \beta)$}
    \label{fig:h-nb-integrals}
  \end{subfigure}%
  ~
  \begin{subfigure}[t]{.33\textwidth}
    \centering
    \image{1}{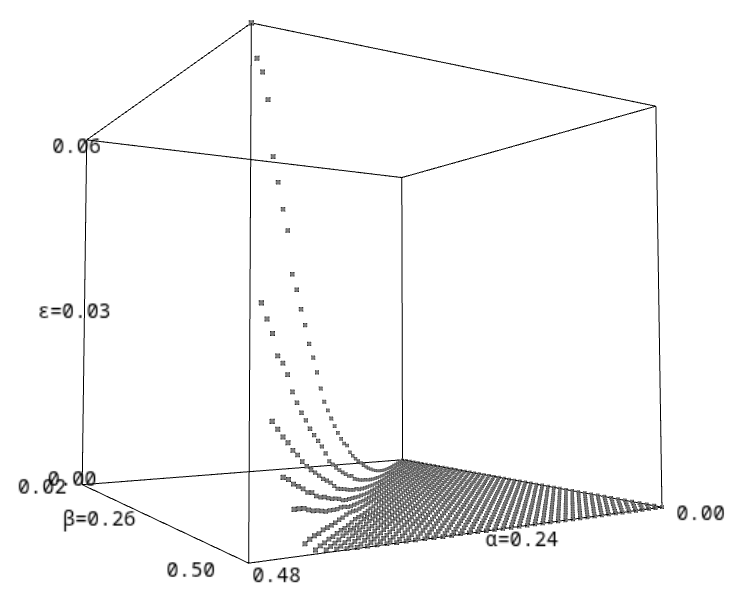}
    \caption{Bound on the precision for the computation of~$h(\alpha, \beta)$}
    \label{fig:h-precision}
  \end{subfigure}%
  ~
  \begin{subfigure}[t]{.33\textwidth}
    \centering
    \image{1}{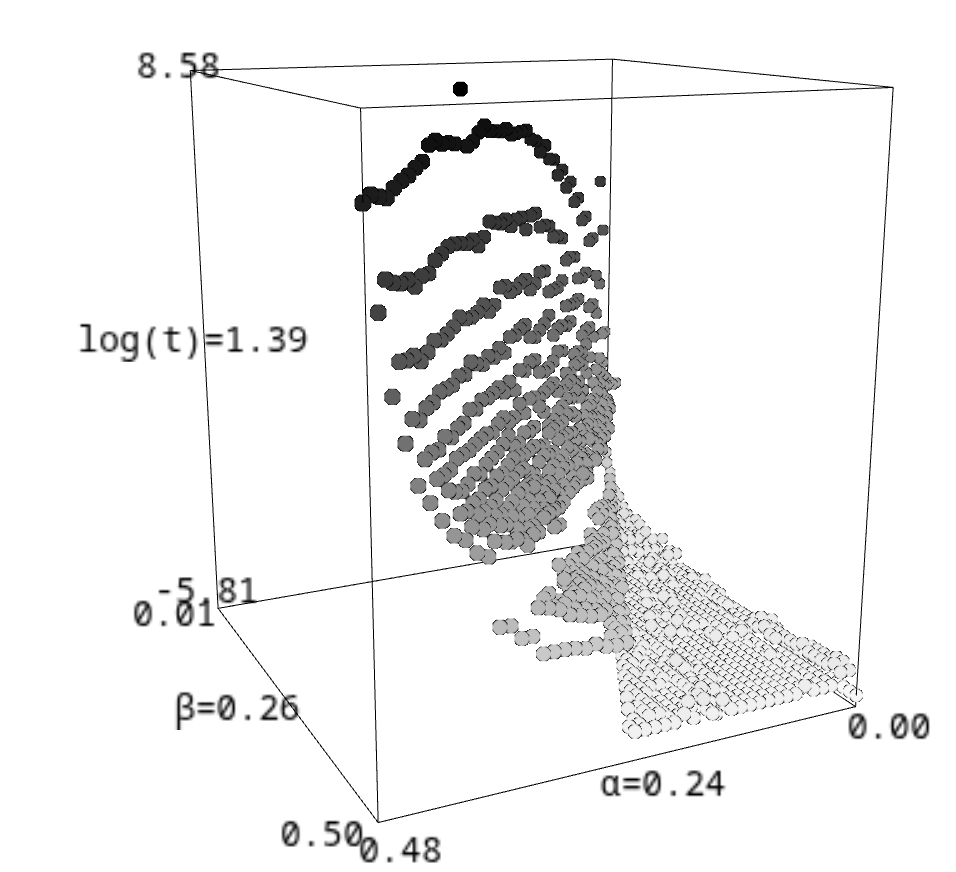}
    \caption{Log of the time of computation of~$h(\alpha, \beta)$}
    \label{fig:h-timing}
  \end{subfigure}
\end{figure}

The computation was requested with~$14$ bits of precision, however the actual returned precision deteriorates as the number of involved integrals, and the number of balanced polytopes required in the computation, increase. Figure~\ref{fig:h-precision} represents the certified precision of the final result. Finally, Figure~\ref{fig:h-timing} represents the computation time, in logarithmic scale. The computation time increases exponentially as~$(\alpha, \beta)$ approaches the diagonal; the largest values took 20 minutes. The computation time undergoes jumps as higher values of~$k$ become involved in the computation (and thus more boolean functions need to be analyzed).

\subsection{The average logarithm of the middle divisor}

In this section we justify Corollary~\ref{cor:mid-div}.
Let~$\hbar(\alpha, \beta) := 1 - h(\alpha, \beta)$. We start by using the symmetry of divisors,
$$ h(\alpha, \tfrac12) = h(\alpha, 1 - \alpha). $$
Changing variables as~$\alpha = \frac{1-\eta}2$ with~$0\leq \eta \leq 1$, we get
$$ c = \frac12 + \int_0^{1/2} \hbar(\alpha, 1-\alpha) \df \alpha = \frac12\big(1 + \int_0^1 \hbar\big(\tfrac{1-\eta}2, \tfrac{1+\eta}2\big) \df \eta\big). $$
Let~$\eta^*>0$ be a small parameter. We split the integral at~$\eta^*$, using the trivial bound~$\hbar(\dotsb) \leq 1$ for those~$\eta\leq \eta^*$.
For those~$\eta>\eta^*$ we use the expression from Theorem~\ref{th:expr-h}. This gives
$$ c = \tfrac12(1 + c_1 + c_2 + c_2') + \eps_1, \qquad \eps_1 \in [0,\eta^*/2], $$
where
\begin{align}
  c_1 ={}& \int_{\eta^*}^1 \Big(B\Big(\frac2{1+\eta} + 1\Big) - B(2)\Big)\df \eta, \notag \\
  c_2 ={}& \sum_{2\leq k \leq 1/\eta^*} \sum_{P\in \cP_k} \int_{\eta^*}^1 \int_P \rho\Big(\frac{1-\sum_j t_j - s}{\eta}\Big) B\Big(\frac{2s}{1+\eta}\Big) \frac{\df t \df s \df \eta}{t_1 \dotsb t_k s}, \label{eq:def-c2} \\
  c_2' ={}& \sum_{2\leq k \leq 1/\eta^*} \sum_{P\in \cP_k'} \int_{\eta^*}^1 \int_P \rho\Big(\frac{1-\sum_j t_j}{\eta}\Big) \frac{\df t \df \eta}{t_1 \dotsb t_k }. \label{eq:def-c2p}
\end{align}
First, since~$1\leq 2/(1+\eta) \leq 2$, we have~$B(2/(1+\eta) + 1) = 1 + \log(2/(1+\eta))$ for~$\eta\in [0, 1]$ and thus~$c_1 = \int_{\eta^*}^1 \log(2/(1+\eta)) \df \eta = 1 - \eta^* - (1 + \eta^*) \log(2/(1 + \eta^*))$.

Next, we claim that~$c_2 = 0$. Indeed the second inequality in the definition~\eqref{eq:def-DeltaF} (with~$\alpha = \tfrac{1-\eta}2$) implies~$s < \frac{1+\eta}2$, so that~$2s/(1+\eta) < 1$, thus since~$B(u)$ vanishes for~$u<1$ the integrand in~$c_2$ vanishes identically.

Finally there remains the computation of~$c_2'$. We pick~$\eta^* = 0.01$, which gives~$c_1 \in [0.29997, 0.29998]$.
First we dismiss the values~$k\geq 8$ using the bounds~\eqref{eq:maj_k_grand}. Let~$E(\eta)$ be the minimum of the two quantities~\eqref{eq:maj_k_grand} evaluated at~$k_0=8$, $\alpha = \frac{1-\eta}2$ and~$\beta = \frac{1 + \eta}2$. Note that~$E(\eta)$ is decreasing. By sampling~$E(\eta)$ at multiples of~$10^{-5}$, we get
$$ c_2' = c''_2 + \eps_2, \qquad \eps_2 \leq \int_{\eta^*}^1 E(\eta) \df \eta \in [0, 0.002509], $$
where~$c''_2$ is the contribution of~$k\leq 7$ in~$c_2'$. Finally we compute the integrals in~$c''_2$ using the method described in Section~\ref{sec:comp-sieve-integr}, requesting 17 bits of precision. We chose different balancing parameters~$\eta$ depending on the dimension: $\eta = \log(1.3)$ for~$3\leq k \leq 5$, $\eta = \log(1.6)$ for~$k=6$ and~$\eta = \log 2$ for~$k=7$ (there is no contribution from~$k=2$). The reason is to be able to dismiss faster the contribution from~$k\geq 5$ by trivial bounds. The computation returned the bound~$0.004163 \leq c''_2 \leq 0.01899$ (with the contribution from~$k\geq 5$ being~$\leq 0.011102$). Gathering the various bounds we deduce successively~$c_2' \in [0.004162, 0.02150]$ and~$c\in[0.6520, 0.6658]$, and thus Corollary~\ref{cor:mid-div} follows.

\bibliographystyle{amsplain2}
\bibliography{sieve-num-integral}

\providecommand{\bysame}{\leavevmode\hbox to3em{\hrulefill}\thinspace}
\providecommand{\MR}{\relax\ifhmode\unskip\space\fi MR }
\providecommand{\MRhref}[2]{%
  \href{http://www.ams.org/mathscinet-getitem?mr=#1}{#2}
}
\providecommand{\href}[2]{#2}
\begin{thebibliography}{10}

\bibitem{BaldoniEtAl2011}
V.~Baldoni, N.~Berline, J.~A. De~Loera, M.~K{\"o}ppe, and M.~Vergne, \emph{How
  to integrate a polynomial over a simplex}, Math. Comput. \textbf{80} (2011),
  no.~273, 297--325.

\bibitem{Barvinok1991}
A.~I. Barvinok, \emph{Computation of exponential integrals}, Zap. Nauchn.
  Semin. Leningr. Otd. Mat. Inst. Steklova \textbf{192} (1991), 149--162.

\bibitem{Brion1988}
M.~Brion, \emph{Lattice points in convex polyhedra.}, Ann. Sci. {\'E}c. Norm.
  Sup{\'e}r. (4) \textbf{21} (1988), no.~4, 653--663.

\bibitem{Bukhshtab1937}
A.~A. Buchshtab, \emph{Asymptotische {Abschätzung} einer allgemeinen
  zahlentheoretischen {Funktion}}, Rec. Math. Moscou, n. Ser. \textbf{2}
  (1937), 1239--1246.

\bibitem{CarliniEtAl2012}
E.~Carlini, M.~V. Catalisano, and A.~V. Geramita, \emph{The solution to the
  {Waring} problem for monomials and the sum of coprime monomials}, J. Algebra
  \textbf{370} (2012), 5--14.

\bibitem{Chen1973}
J.~Chen, \emph{On the representation of a larger even integer as the sum of a
  prime and the product of at most two primes}, Sci. Sin. \textbf{16} (1973),
  157--176.

\bibitem{DavisRabinowitz1984}
P.~J. Davis and P.~Rabinowitz, \emph{Methods of numerical integration. 2nd ed},
  second edition ed., Computer Science and Applied Mathematics, Academic Press,
  Inc., 1984.

\bibitem{DeLoeraEtAl2011}
J.~A. De~Loera, B.~Dutra, M.~K{\"o}ppe, S.~Moreinis, G.~Pinto, and J.~Wu,
  \emph{Software for exact integration of polynomials over polyhedra}, ACM
  Commun. Comput. Algebra \textbf{45} (2011), no.~3, 169--172.

\bibitem{DeLoeraEtAl2013}
\bysame, \emph{Software for exact integration of polynomials over polyhedra},
  Comput. Geom. \textbf{46} (2013), no.~3, 232--252.

\bibitem{Dickman1930}
K.~Dickman, \emph{On the frequency of numbers containing prime factors of a
  certain relative magnitude.}, Ark. Mat. Astron. Fys. \textbf{22 A} (1930),
  no.~10, 14.

\bibitem{github-sieve-integral}
S.~Drappeau and A.~Mounier, \emph{Computation of sieve integrals using
  {LattE}}, GitHub repository, 2026, Available at
  \url{https://github.com/sarydrappeau/sieve_integral}.

\bibitem{Erdos1960}
P.~Erd{\H{o}}s, \emph{On an asymptotic inequality in number theory}, Vestn.
  Leningr. Univ., Mat. Mekh. Astron. \textbf{15} (1960), no.~3, 41--49.

\bibitem{Ford2008}
K.~Ford, \emph{The distribution of integers with a divisor in a given
  interval}, Ann. of Math. (2) \textbf{168} (2008), no.~2, 367--433.

\bibitem{FordMaynard2024}
K.~Ford and J.~Maynard, \emph{On the theory of prime producing sieves}, arXiv
  preprint, 2024.

\bibitem{FriedlanderIwaniec2010}
J.~Friedlander and H.~Iwaniec, \emph{Opera de cribro}, Colloq. Publ., Am. Math.
  Soc., vol.~57, Providence, RI: American Mathematical Society (AMS), 2010.

\bibitem{GreenSawhney2026}
B.~Green and M.~Sawhney, \emph{The proportion of permutations fixing a
  $k$-set}, arXiv preprint.

\bibitem{GrimmeltMerikoski2025}
L.~Grimmelt and J.~Merikoski, \emph{On the greatest prime factor and uniform
  equidistribution of quadratic polynomials}, arXiv preprint, 2025.

\bibitem{Haddad2026}
T.~Haddad, \emph{Poisson-{Dirichlet} approximation for counting integers with
  divisors in an interval}, arXiv preprint, 2026.

\bibitem{HaddadKoukoulopoulos2025}
T.~Haddad and D.~Koukoulopoulos, \emph{On {Arratia}'s coupling and the
  {Dirichlet} law for the factors of a random integer}, J. {\'E}c. Polytech.,
  Math. \textbf{12} (2025), 1565--1604.

\bibitem{HalberstamRichert1974}
H.~Halberstam and H.-E. Richert, \emph{Sieve methods}, Academic Press,
  London-New York, 1974, London Mathematical Society Monographs, No. 4.

\bibitem{Harman1983}
G.~Harman, \emph{On the distribution of {$\alpha p$} modulo one}, J. Lond.
  Math. Soc., II. Ser. \textbf{27} (1983), 9--18.

\bibitem{Harman2007}
\bysame, \emph{Prime-detecting sieves}, Lond. Math. Soc. Monogr. Ser., vol.~33,
  Princeton, NJ: Princeton University Press, 2007.

\bibitem{Knuth2011}
D.~E. Knuth, \emph{The art of computer programming. {Volume} {4A}.
  {Combinatorial} algorithms. {Part} 1.}, Upper Saddle River, NJ:
  Addison-Wesley, 2011.

\bibitem{KnuthPardo1977}
D.~E. Knuth and L.~T. Pardo, \emph{Analysis of a simple factorization
  algorithm}, Theor. Comput. Sci. \textbf{3} (1977), 321--348.

\bibitem{Lasserre1983}
J.~B. Lasserre, \emph{An analytical expression and an algorithm for the volume
  of a convex polyhedron in {{\(R^ n\)}}.}, J. Optim. Theory Appl. \textbf{39}
  (1983), 363--377.

\bibitem{Lawrence1991}
J.~Lawrence, \emph{Polytope volume computation}, Math. Comput. \textbf{57}
  (1991), no.~195, 259--271.

\bibitem{Li2024}
R.~Li, \emph{On the largest prime factor of quadratic polynomials}, arXiv
  preprint, 2024.

\bibitem{Li2025}
\bysame, \emph{The number of primes in short intervals and numerical
  calculations for {Harman}'s sieve}, arXiv preprint, 2025.

\bibitem{Li2026}
\bysame, \emph{On {Chen}'s theorem, {Goldbach}'s conjecture and almost prime
  twins {II}}, Math. Rep., Buchar. \textbf{28} (2026), no.~1-2, 39--61.

\bibitem{MarsagliaEtAl1989}
G.~Marsaglia, A.~Zaman, and J.~Marsaglia, \emph{Numerical solution of some
  classical differential-difference equations}, Math. Comput. \textbf{53}
  (1989), no.~187, 191--201.

\bibitem{Matomaeki2009}
K.~Matom{\"a}ki, \emph{The distribution of {{\(\alpha p\)}} modulo one}, Math.
  Proc. Camb. Philos. Soc. \textbf{147} (2009), no.~2, 267--283.

\bibitem{MatomaekiZuniga-Alterman2025}
K.~Matom{\"a}ki and S.~Zuniga-Alterman, \emph{Weighted sieves with switching},
  Math. Proc. Camb. Philos. Soc. \textbf{179} (2025), no.~2, 351--372.

\bibitem{Maynard2019}
J.~Maynard, \emph{Primes with restricted digits}, Invent. Math. \textbf{217}
  (2019), no.~1, 127--218, Mathematica code available at
  \url{https://arxiv.org/src/1604.01041}.

\bibitem{Merikoski2023}
J.~Merikoski, \emph{On the largest prime factor of {{\(n^2 + 1\)}}}, J. Eur.
  Math. Soc. (JEMS) \textbf{25} (2023), no.~4, 1253--1284.

\bibitem{Mounier2025}
A.~Mounier, \emph{An effective lower bound sieve for friable numbers}, Acta
  Arith. \textbf{220} (2025), no.~2, 121--159.

\bibitem{OEIS}
{OEIS Foundation Inc.}, \emph{The {O}n-{L}ine {E}ncyclopedia of {I}nteger
  {S}equences}, 2025, Published electronically at \url{http://oeis.org}.

\bibitem{Stadlmann2022}
J.~Stadlmann, \emph{On the mean square gap between primes}, arXiv preprint,
  2022.

\bibitem{TakahasiMori1974}
H.~Takahasi and M.~Mori, \emph{Double exponential formulas for numerical
  integration}, Publ. Res. Inst. Math. Sci. \textbf{9} (1974), 721--741.

\bibitem{Tenenbaum1980}
G.~Tenenbaum, \emph{Lois de répartition des diviseurs, 2}, Acta Arith.
  \textbf{38} (1980), no.~1, 1--36.

\bibitem{Tenenbaum1984}
\bysame, \emph{Sur la probabilité qu’un entier possède un diviseur dans un
  intervalle donné}, Compos. Math. \textbf{51} (1984), 243--263.

\bibitem{Tenenbaum2015}
\bysame, \emph{Introduction to analytic and probabilistic number theory.
  {Transl}. from the 3rd {French} edition by {Patrick} {D}. {F}. {Ion}}, 3rd
  expanded ed. ed., Grad. Stud. Math., vol. 163, Providence, RI: American
  Mathematical Society (AMS), 2015.

\bibitem{PARI-GP}
{The PARI~Group}, Univ. Bordeaux, \emph{{PARI/GP version \texttt{2.15.4}}},
  2023, available from \url{http://pari.math.u-bordeaux.fr/}.

\bibitem{SageMath}
{The Sage Developers}, \emph{{S}agemath, the {S}age {M}athematics {S}oftware
  {S}ystem ({V}ersion 10.7)}, 2025-08-09, {\tt https://www.sagemath.org}.

\bibitem{Mathematica}
{Wolfram Research, Inc.}, \emph{Mathematica, {V}ersion 14.3}, Champaign, IL,
  2025.

\end{thebibliography}

\end{document}